\documentclass[final]{siamart}
\usepackage[margin=1in]{geometry}
\usepackage{amsmath,amssymb,amsfonts, stmaryrd}
\usepackage{hyperref}		\usepackage{color}
\usepackage{graphicx}
\usepackage{array}
\usepackage{subcaption}
\usepackage{bigdelim}
\usepackage{xfrac}
\usepackage{multirow}
\usepackage{siunitx}  \usepackage{cite}

\usepackage{makecell}
\usepackage{hhline}
\usepackage{cellspace}
\usepackage{mathtools}
\DeclarePairedDelimiterX\set[1]\{\}{\nonscript\,#1\nonscript\,}
\newsiamremark{remark}{Remark}

\newcommand{\tcr}{}

\newcommand{\TheTitle}{Nonsymmetric Algebraic Multigrid Based\\on Local Approximate Ideal Restriction ($\ell$AIR)}
\newcommand{\TheAuthors}{T. A. Manteuffel, J. Ruge, and B. S. Southworth}
\headers{AMG Based on Local Approximate Ideal Restriction}{\TheAuthors}
\title{{\TheTitle}\thanks{This research was conducted with Government support under and awarded by DoD, Air Force Office of Scientific Research, National Defense Science and Engineering Graduate (NDSEG) Fellowship, 32 CFR 168a; and under
the auspices of the U.S. Department of Energy under grant numbers (SC) DE-FC02-03ER25574 and (NNSA) DE-NA0002376, and
  Lawrence Livermore National Laboratory under contract B614452.
  }}

\author{  Thomas~A.~Manteuffel
  \thanks{Department of Applied Mathematics,
          University of Colorado at Boulder.}
  \and
  John Ruge
  \footnotemark[2]
  \and
  Ben~S.~Southworth
  \thanks{Department of Applied Mathematics,
          University of Colorado at Boulder
          (\email{ben.s.southworth@gmail.com}).}
}

\ifpdf\hypersetup{  pdftitle={\TheTitle},
  pdfauthor={\TheAuthors}
}
\fi

\begin{document}
\maketitle

\begin{abstract}
Algebraic multigrid (AMG) solvers and preconditioners are some of the fastest numerical methods to solve linear systems,
particularly in a parallel environment, scaling to hundreds of thousands of cores. Most AMG methods and theory
assume a symmetric positive definite operator. This paper presents a new variation on classical AMG for
nonsymmetric matrices (denoted $\ell$AIR), based on a local approximation to the ideal restriction operator, coupled with F-relaxation.
A new block decomposition of the AMG error-propagation operator is used for a spectral analysis of convergence, and the efficacy
of the algorithm is demonstrated on systems arising from the discrete form of the advection-diffusion-reaction equation. $\ell$AIR is
shown to be a robust solver for various discretizations of the advection-diffusion-reaction equation, including time-dependent and
steady-state, from purely advective to purely diffusive. Convergence is robust for discretizations on unstructured meshes and using
higher-order finite elements, and is particularly effective on upwind discontinuous Galerkin discretizations. Although the implementation used
here is not parallel, each part of the algorithm is highly parallelizable, avoiding common multigrid adjustments for strong advection such as
line-relaxation and K- or W-cycles that can be effective in serial, but suffer from high communication costs in parallel, limiting their scalability.
\end{abstract}

\section{Introduction and motivation}\label{sec:intro}

Algebraic multigrid (AMG) is an iterative solver for large sparse linear systems \cite{McCormick:1982cd}. For the
symmetric positive definite (SPD) case, often resulting from the discretization of elliptic partial differential equations (PDEs),
convergence of AMG is well-motivated \cite{ Falgout:2005hm,Falgout:2004cs,MacLachlan:2014di,McCormick:1982cd,Notay:2014uc,
Vassilevski:2008wd} and AMG is among the fastest numerical solvers available. Furthermore, AMG scales in parallel to hundreds of
thousands of cores \cite{Baker:2012ko}, making it a key component of many high-performance simulation codes. 

A variety of AMG methods have been proposed to generalize the AMG framework to nonsymmetric matrices. Perhaps
the original idea, and still a common approach, is to treat a nonsymmetric matrix as if it were symmetric, where restriction, $R$, is given by the
transpose of interpolation, $P$: $R := P^T$ \cite{ruge:1987}. In some circumstances, this is an effective choice, but when and why this is effective
is a question that relies largely on experience. Such an approach is the extent of published research on classical AMG pointwise interpolation
formulae \cite{McCormick:1982cd,ruge:1987} for nonsymmetric problems.

In contrast, a number of works suggest that restriction should be built based on $A^T$, or the column-space of $A$, and interpolation should
be based on $A$, or the row-space of $A$ (for example, \cite{Brezina:2010dm,Sala:2008cv}). For classical pointwise interpolation formulae, this
is not immediately applicable because they are based on
an appropriate strength-of-connection measure and splitting of nodes into C-points and F-points (CF-splitting). Using a different pointwise interpolation
formula for $P$ and $R$  based on $A$ and $A^T$ would theoretically require an independent CF-splitting for each, introducing additional
difficulties such as non-square coarse-grid operators. Of course, one could assume a fixed CF-splitting and use formulae for $P$ and $R$
based on $A$ and $A^T$, respectively, but simple tests indicate that this is not effective for nonsymmetric problems.
Aggregation-based AMG is more applicable to using information from $A$ and $A^T$, leading to several variations in classical
smoothed aggregation (SA) for nonsymmetric problems \cite{Brezina:2010dm,Guillard:1996wz,Sala:2008cv}. There have also been recent
solvers developed that use a mix of CF-splitting and aggregation-concepts and are applicable to nonsymmetric problems, typically using some
form of constrained minimization on $P$ and $R$ to approximate the so-called ``ideal'' operators  \cite{Lottes:2017jz,Manteuffel:2017,Olson:2011fg,
Wiesner:2014cy}. Although some results on this front have been encouraging, a robust AMG solver for nonsymmetric linear systems remains
an open problem. 

Part of the reason that nonsymmetric solvers are less robust than their SPD counterparts is that little is known about convergence theory
of AMG in the nonsymmetric setting, thus limiting theoretical motivation to develop new methods. Strong theoretical results on convergence
of iterative methods for non-SPD matrices are difficult to establish and few and far between in the literature. Multigrid traditionally
measures convergence in the matrix-induced energy norm, $\|\mathbf{x}\|_A^2 = \langle A\mathbf{x},\mathbf{x}\rangle$, which is not a valid norm
in the non-SPD setting. A generalization of the energy norm to a $\sqrt{A^*A}$-norm was introduced in \cite{Brezina:2010dm}, and two-grid
convergence for an aggregation-based solver proved under the (strong) assumption of stability of the non-orthogonal coarse-grid correction.
Further analysis of the stability assumption and more practical conditions for two-level and multilevel convergence are ongoing work \cite{nonsymm}. In any
case, the $\sqrt{A^*A}$-norm suggests that $R$ be based on left singular vectors and $P$ based on right singular vectors, but the norm and
corresponding orthogonal coarse-grid correction are intractable to compute. An asymptotic bound on error propagation can be
found by considering the spectral radius of error propagation \cite{Lottes:2017jz,amgir,Notay:2010em}, but asymptotic bounds are not always
indicative of practical performance, and even spectral analyses can be difficult because many tools of linear algebra are not applicable; for
example, the eigenvectors do not necessarily form a basis for the space and so error cannot be expanded in terms of eigenvectors. Recently
it was suggested that the field of values is more indicative of practical performance \cite{notay2018}, a result that needs further
consideration in the practical setting. In \cite{amgir}, multilevel
error-propagation was shown to be nilpotent in the case of block triangular matrices, but such structure makes up only a specific class of
nonsymmetric problems. More generally, \cite{Lottes:2017jz} proposed a relatively self-contained framework for nonsymmetric matrices with
positive real part, $(A+A^T) > 0$, showing two-grid convergence of classical-style AMG (that is, interpolating and restricting C-points by value)
in a spectral sense as well as in an appropriately derived (albeit difficult to compute) norm. The resulting theory was based on approximating
the action of ideal interpolation and ideal restriction on vectors, with accuracy of approximation for a given vector based on the so-called
\textit{form absolute value} \cite{Lottes:2017jz}. However, because the form absolute value is infeasible to compute, the practical solver motivated
in \cite{Lottes:2017jz} reduced to a generalization of constrained energy-minimization in the nonsymmetric setting, similar to
\cite{Wiesner:2014cy,Manteuffel:2017,Olson:2011fg}. 

In terms of solver development, a good model problem to study AMG for nonsymmetric systems is the advection-diffusion-reaction equation,
which comes up in fluid flow and particle transport equations, among others. Let $\kappa$ be the diffusion coefficient, $\beta$ a measure
of the size of advection, and $h$ the mesh spacing. Then $R_h:=\frac{\beta h}{\kappa}$, often called the grid Reynold's number, is a measure of
the numerical balance between advection and diffusion. For $R_h > 1$, the problem is advection-dominated. Under appropriate
boundary conditions, the limit of the weak form as $\kappa \to 0$ is the purely hyperbolic steady-state transport equation, which is well-discretized
by upwinding, resulting in a triangular or block-triangular matrix. Conversely, for $R_h < 1$, the resulting discretization is
diffusion dominated, converging to an operator that numerically looks like a diffusion discretization as $\kappa$ grows (i.e., the advection component
is arbitrarily small and the matrix is effectively SPD). AMG (and many other iterative methods) are designed for elliptic, diffusion-like problems,
and in many cases achieve excellent convergence rates for the diffusion-dominated case. 
Recently, a reduction-based AMG method was developed that is highly effective on the hyperbolic limit \cite{amgir}, but deteriorates when
significant diffusion is introduced. The goal of this work is to bridge this gap and develop an AMG solver that is robust across the spectrum of
diffusivity and, in particular, well-suited for a high-performance parallel implementation. Note that a scalar model problem is chosen here
intentionally to isolate the effects of nonsymmetry. Linear systems with block structure resulting from a system of PDEs can be
difficult for AMG, often requiring individual attention that is outside the scope of this work (for example, \cite{Clees:2005uy}). 

There have been many efforts at developing robust iterative methods specifically for advection-diffusion-type problems. Geometric multigrid (GMG)
methods for problems with advection are often based on the concept of semi-coarsening or line-relaxation in the direction of advection
\cite{Oosterlee:1998ih, Schaffer:1998ui,Yavneh:1998fw,Yavneh:2012fb}; however, such approaches have several drawbacks, including
requiring a priori knowledge of the underlying problem and discretization, as well as limitations with respect to unstructured meshes and
parallelism. Overall, robust convergence was obtained in \cite{Yavneh:2012fb} for a range of advection and diffusion, but the method employed
a line Gauss-Seidel smoother in the $x$- and $y$-direction. \tcr{In the parallel setting, each line relaxation has an $O(\log P)$ communication cost,
for $P$ processors, and also requires knowledge of the problem geometry \cite{luke18}.}
Several efforts have been made to develop AMG algorithms that are robust for advection-dominated problems \cite{Bayramov:2017gm,
Guillard:1996wz,Kim:2003jj,Notay:2012ft,Wu:2006jn}. Generally strong convergence was obtained in \cite{Bayramov:2017gm} and
\cite{Notay:2012ft}; however, the results in each were based on multigrid K- or W-cycles, which come with a substantially higher
communication cost in parallel than a normal V-cycle, thus limiting their parallel efficiency. Analysis has also
been done in the scope of Krylov methods, for example, \cite{Liesen:2005ci,Sonneveld:2009gh}, and other algebraic solvers such as
\cite{Bertaccini:2004cv} and \cite{Krukier:2015ho}. However, a solver that is robust across the spectrum of diffusivity and, in particular,
scalable and parallelizable, has proven difficult to achieve. 

Here, we derive a new way to look at the roles of interpolation and restriction operators in AMG, focusing on the importance of
approximating the so-called ``ideal restriction'' operator.
Ideal restriction was also the motivation behind the $n$AIR algorithm developed in \cite{amgir} for upwind discretizations of hyperbolic PDEs,
as well as a number of works using an approximate block LDU decomposition \cite{Carvalho:2001ib,Mandel:1990kv,Mense:2008gj,Notay:2000vy,
Wiesner:2014cy}. $n$AIR, in particular, proved to be a highly effective solver on strictly advective equations
such as the steady-state transport equation, but the method was specifically developed for matrices with a block-triangular or near block-triangular
structure. In this paper, we generalize the approach of \cite{amgir} to approximating ideal restriction for arbitrary matrices. Numerical
tests on discretizations of advection-diffusion, from strictly advective to strictly diffusive, are then used to examine how classical AMG
interpolation methods fit into the framework of a local approximate ideal restriction ($\ell$AIR). Together, new analytical and numerical results lead to the
development of a solver, $\ell$AIR, that is robust for advection-diffusion-reaction equations discretized on an unstructured mesh, with moving
flow, and any ratio of advection to diffusion. In the limiting case of a diffusion problem, the new solver is competitive with classical AMG
methods; in the hyperbolic limit of no diffusion, the new solver is almost equivalent to that presented in\cite{amgir}. Moreover, each component of $\ell$AIR is
highly parallelizable, and strong convergence does not require sequential algorithms such as Gauss-Seidel line-relaxation or high-communication
acceleration methods such as K- or W-cycles. 

Section \ref{sec:background} formally introduces AMG and its key components, followed by a discussion of the ideal interpolation and
restriction operators and a new block analysis of the AMG error-propagation operator. This new look at error propagation helps explain
the role of restriction in AMG and how it is coupled with interpolation and the splitting of degrees-of-freedom (DOFs) into ``coarse points''
(C-points) and ``fine points'' (F-points). A general approach to building restriction operators in AMG is then introduced in Section \ref{sec:air}.
Numerical results in Section \ref{sec:results} demonstrate that $\ell$AIR is a robust solver for scalar advection-diffusion-type equations, and that
a classical AMG framework (with an appropriate restriction operator) can be extended to nonsymmetric problems. In particular, for
advection-dominated problems, $\ell$AIR outperforms nonsymmetric smoothed aggregation methods, which are generally the most robust AMG solvers for
nonsymmetric problems. An overview of results and future directions are presented in Section \ref{sec:conc}.
$\ell$AIR is implemented in the PyAMG library \cite{Bell:2008} and its parallel implementation in \textit{hypre} \cite{Falgout:2002vu} is ongoing work.

\section{Transfer operators and nonsymmetric algebraic multigrid}\label{sec:background}

\subsection{Algebraic multigrid}\label{sec:background:amg}

Algebraic multigrid is a multilevel iterative solver based on two processes: relaxation and coarse-grid correction. Relaxation
is typically chosen as some simple iterative method such as Jacobi or Gauss-Seidel, and is expected to eliminate error associated
with large singular values of the matrix. For differential operators, such singular values are often geometrically high frequency modes,
on which classical methods such as Jacobi tend to be effective. To complement relaxation, coarse-grid correction consists of a
subspace correction designed to attenuate error associated with small eigenvalues of the matrix. If the so-called \textit{coarse-grid
operator} is too large to explicitly invert, AMG is applied recursively. Together, relaxation and coarse-grid correction are applied
iteratively until a desired residual tolerance is achieved.

Given a linear system $A\mathbf{x} = \mathbf{b}$, the first step in building the solver or ``hierarchy'' is
to partition the ``points'' or degrees-of-freedom (DOFs) into C-points and F-points, where C-points represent DOFs on the
coarse grid. Then, $A$ can be symbolically ordered in block form:
\begin{align}
A & = \begin{pmatrix} A_{ff} & A_{fc} \\ A_{cf} & A_{cc}\end{pmatrix},\label{eq:acf}
\end{align}
where F-points are ordered first, followed by C-points.  Let $A\in\mathbb{R}^{n\times n}$, $n_c$ be the number of C-points, and
$n_f$ the number of F-points. Next, interpolation and restriction operators are defined,
$P:\mathbb{R}^{n_c}\mapsto\mathbb{R}^n$ and $R:\mathbb{R}^{n}\mapsto\mathbb{R}^{n_c}$, respectively, that map between
the current space and the coarse space. Here, we assume that C-points are interpolated and restricted by injection in the classical
AMG sense \cite{McCormick:1982cd,ruge:1987}; that is, $P$ and $R$ take the following block form:
\begin{align}\label{eq:PR}
P & = \begin{pmatrix} W \\ I\end{pmatrix}, \hspace{3ex} R = \begin{pmatrix} Z & I \end{pmatrix},
\end{align}
where the identity block makes up C-point rows of $P$ and C-point columns of $R$. Finally, the coarse-grid operator is defined
as $\mathcal{K} := RAP$, and the two-level coarse-grid correction is given as
\begin{align*}
\mathbf{x}^{(i+1)} = \mathbf{x}^{(i)} + P(RAP)^{-1}R(\mathbf{b} - A\mathbf{x}^{(i)}).
\end{align*}

Let $\overline{\mathbf{x}}$ be the exact solution to $A\mathbf{x} = \mathbf{b}$ and $\mathbf{x}^{(i)}$ some approximation. Two measures
of convergence used in iterative methods are the error, $\mathbf{e}^{(i)} := \overline{\mathbf{x}} - \mathbf{x}^{(i)}$, and residual,
$\mathbf{r}^{(i)} := \mathbf{b}- A\mathbf{x}^{(i)} = A \mathbf{e}^{(i)}$. Although eliminating error is typically the true goal of an iterative method, for
nonsingular $A$, $\mathbf{e} = \mathbf{0}$ if and only if $\mathbf{r} = \mathbf{0}$. Since $\mathbf{r}$ is measurable in practice
and $\mathbf{e}$ is not, it is worth considering both. Error-propagation and residual-propagation of coarse-grid correction take
on the following operator forms:
\begin{align}
\mathbf{e}^{(i+1)} & = \mathcal{E}\mathbf{e}^{(i)} := (I - P(RAP)^{-1}RA)\mathbf{e}^{(i)} \label{eq:err_prop}, \\
\mathbf{r}^{(i+1)} & = \mathcal{R}\mathbf{e}^{(i)} := (I - AP(RAP)^{-1}R)\mathbf{r}^{(i)}. \label{eq:res_prop}
\end{align}
Note that these operators are similar: $\mathcal{E} = A^{-1}\mathcal{R}A$. In the case of a symmetric matrix, $A$, and
Galerkin coarse grid, $R:=P^T$, then $\mathcal{E} = \mathcal{R}^T$. In general, for symmetric $A$, the roles of restriction and
interpolation are more-or-less interchangeable. Suppose $A$ is symmetric and $P_0$ and $R_0$ are effective transfer operators
at reducing the error in the $\ell^2$-norm: $1 \approx \|I - P(RAP)^{-1}RA\| = \|I - AR^T(P^TAR^T)^{-1}P^T\|$. Then,
$R_1:= P_0^T$ and $P_1:=R_0^T$ are effective transfer operators in reducing the residual. A similar result based on
symmetry holds for the spectral radius. Although convergence is not necessarily
measured in the $\ell^2$- or spectral-sense, this is indicative that the same principles can be effective for building $R$ \textit{and} $P$ for
symmetric matrices. For nonsymmetric matrices, the relation between restriction and interpolation is less clear. In particular, for a
nonsymmetric problem, if an AMG solver based on $P_0$ and $R_0$ effectively reduces the error, then an AMG solver based on
$P_1:=R_0^T$ and $R_1:=P_0^T$ will effectively reduce the residual with respect to $A^T$. However, this does \textit{not} indicate
that a solver based on $P_1$ and $R_1$ will be effective when applied to $A$. 

This work focuses on so-called ideal interpolation and ideal restriction operators. When considered in a reduction setting
\cite{Kamowitz:1987eq,MacLachlan:2006gt,amgir,Ries:1983cq},
each of these operators is tightly coupled with an exact solve or effective relaxation on F-points. The idea behind
F-relaxation is to improve the solution at F-points, and then distribute this accuracy to C-points via coarse-grid correction
(ideal interpolation), or get an accurate coarse-grid correction at C-points and distribute this accuracy to F-points via
F-relaxation (ideal restriction). Here we consider some F-relaxation scheme where $\Delta$ is an approximation to
$A_{ff}^{-1}$. Assuming F-points have been chosen such that $A_{ff}$ is
well-conditioned, then F-relaxation should be effective at reducing F-point residuals and/or errors.
Residual propagation and error propagation for F-relaxation are given respectively by
\begin{align}
\mathbf{e}^{(i+1)} & = \begin{pmatrix}I - \Delta A_{ff} & -\Delta A_{fc} \\ 0 & I\end{pmatrix}
	\begin{pmatrix} \mathbf{e}_f^{(i)} \\ \mathbf{e}_c^{(i)} \end{pmatrix}, \label{eq:jacobi_err} \\
\mathbf{r}^{(i+1)} & = \begin{pmatrix}I - A_{ff}\Delta & 0 \\ -A_{cf}\Delta & I\end{pmatrix}
	\begin{pmatrix} \mathbf{r}_f^{(i)} \\ \mathbf{r}_c^{(i)} \end{pmatrix} \label{eq:jacobi_res}. 
\end{align}
For symmetric $A$, residual and error propagation are adjoints of each other. \tcr{For nonsymmetric matrices,
we can note that in relaxing \textit{only} on $A_{ff}$, error propagation and residual propagation
are similar in the F-F block: $I - \Delta A_{ff} = \Delta(I - A_{ff}\Delta)\Delta^{-1}$. 
However, the connection between error and residual reduction
in general is less clear when considering C-points and F-points, as in \eqref{eq:jacobi_err} and \eqref{eq:jacobi_res}.}
Residual reduction is based on the column scaling of $A$ and error reduction on the row-scaling. For numerical results presented
in Section \ref{sec:results}, $\Delta$ corresponds to 1--2 iterations of Jacobi F-relaxation.

\subsection{Ideal interpolation and residual reduction}\label{sec:background:ii}

Suppose that the error after relaxation is in the range of interpolation, that is, $\mathbf{e}^{(i)} = P\mathbf{v}_c$
for some coarse-grid vector $\mathbf{v}_c$. Then coarse-grid correction yields
\begin{align*}
\mathbf{e}^{(i+1)} = \mathbf{e}^{(i)} - P(RAP)^{-1}RA(P\mathbf{v}_c) = \mathbf{e}^{(i)} - P\mathbf{v}_c = \mathbf{0}.
\end{align*}
Obviously, it is advantageous for relaxation to put error in the range of interpolation or, conversely, for interpolation to
accurately represent relaxed error. This is the motivation for the interpolation definition used in classical AMG \cite{McCormick:1982cd,ruge:1987}.
A basic assumption in AMG is that C-points and F-points are chosen such that F-point relaxation can efficiently reduce
the residual at F-points, that is, $A_{ff}$ is well-conditioned, which is also the basis of compatible relaxation \cite{Brannick:2010hz,Livne:2004vt}.
Returning to the error, let $\mathbf{e}_c$ and $\mathbf{e}_f$ be the current error restricted to C-points and F-points,
respectively, and $\mathbf{r} = \mathbf{b} - A\mathbf{x}$ the current residual. Since $A_{ff}$ is assumed to be well-conditioned,
at convergence of F-relaxation we have $\mathbf{r}_f = \mathbf{0}$, which implies 
\begin{align*}
A_{ff} \mathbf{e}_f + A_{fc} \mathbf{e}_c = \mathbf{0} \hspace{2ex}\implies\hspace{2ex}\mathbf{e}_f = -A_{ff}^{-1}A_{fc}\mathbf{e}_c.
\end{align*}
This is the basis for so-called \textit{ideal interpolation,} 
\begin{align}
P_{\textnormal{ideal}} & = \begin{pmatrix} -A_{ff}^{-1}A_{fc} \\ I \end{pmatrix},\label{eq:p_ideal}
\end{align}
which exactly represents the error after exact F-point relaxation. 

Noting that ideal interpolation is based on an assumption of zero residuals at F-points, consider the effect of ideal interpolation
on residual propagation of coarse-grid correction \eqref{eq:res_prop}:
\begin{align*}
\mathbf{r}_f^{(i+1)} & = \mathbf{r}_f^{(i)} - (A_{ff}W+A_{fc})(RAP)^{-1}(Z\mathbf{r}_f^{(i)} + \mathbf{r}_c^{(i)}), \\
\mathbf{r}_c^{(i+1)} & = \mathbf{r}_c^{(i)} - (A_{cf}W + A_{cc})(RAP)^{-1}(Z\mathbf{r}_f^{(i)} + \mathbf{r}_c^{(i)}). 
\end{align*}
For $P = P_{\textnormal{ideal}}$, $W := -A_{ff}^{-1}A_{fc}$, and $RAP = A_{cc}-A_{cf}A_{ff}^{-1}A_{fc}$, independent
of $Z$ \cite{amgir}. Then, \eqref{eq:res_prop} reduces to
\begin{align}
\mathbf{r}^{(i+1)} & = \begin{pmatrix}\mathbf{r}_f^{(i)} \\ -Z\mathbf{r}_f^{(i)}\end{pmatrix}. \label{eq:pideal_res}
\end{align}
That is, ideal interpolation (i) eliminates the contribution of coarse-grid correction to the F-point residual, and (ii) eliminates
the contribution of the previous C-point residual to the updated residual. This is consistent with the notion of preceding coarse-grid
correction based on ideal interpolation with an exact F-point solve: we use F-relaxation to make $\mathbf{r}_f$ small (or zero
for an exact solve) and follow with coarse-grid correction that does not change $\mathbf{r}_f$, but updates $\mathbf{r}_c$ with the
new $\mathbf{r}_f$.

Looking at \eqref{eq:pideal_res} suggests that $Z = \mathbf{0}$ is a good choice for restriction when coupled with $P_{\textnormal{ideal}}$,
as the residual at C-points is then eliminated with coarse-grid correction. In fact, if $Z = \mathbf{0}$, $RAP = A_{cf}W + A_{cc}$ and \eqref{eq:res_prop}
results in an $\ell^2$-orthogonal coarse-grid correction, meaning that $Z = \mathbf{0}$ is optimal in an $\ell^2$-sense. 
Arguably the
fundamental difficulty with nonsymmetric AMG is that coarse-grid correction is typically a non-orthogonal projection. Because non-orthogonal
projections have norm larger than one, this means that coarse-grid correction can actually increase the error or residual, which
can lead to divergent algorithms, particularly if this norm is $h$-dependent. In this sense, $Z = \mathbf{0}$ seems an appealing choice.
However, from a practical perspective, $Z = \mathbf{0}$ is often not the best choice because ideal interpolation (and, thus, an $\ell^2$-orthogonal
coarse-grid correction) is typically not obtained in practice. 

Let $h$ denote the width of a mesh element. The following lemma proves that, if $P \neq P_{\textnormal{ideal}}$ and $Z = \mathbf{0}$, an
increasingly accurate approximation must be made to ideal interpolation to bound $\|\mathcal{R}\|$ as $h\to 0$, which is not scalable. 

\begin{lemma}\label{lem:pbad}
Let $P$ and $R$ take the block form in \eqref{eq:PR}, where $Z = \mathbf{0}$ and $W = -\Delta A_{fc}$ is some approximation 
to $P_{\textnormal{ideal}}$ \eqref{eq:p_ideal}. Define $\widetilde{\Pi} = A\Pi A^{-1}$, corresponding to residual propagation of coarse-grid correction.
Then,
\begin{align*}
\|\widetilde{\Pi}\|^2 = 1 + \sigma_1,
\end{align*}
where $\sigma_1$ is the minimum singular value of $(I - A_{ff}\Delta)A_{fc}\mathcal{K}^{-1}$.
\end{lemma}
\begin{proof}
Define the minimal canonical angle, $\theta_{min}$, between subspaces $\mathcal{X},\mathcal{Y}$, as
\begin{align}
\cos\left(\theta_{min}^{[\mathcal{X},\mathcal{Y}]}\right) & = \sup_{\substack{\mathbf{x}\in\mathcal{X},\\ \mathbf{y}\in\mathcal{Y} } }
	\frac{|\langle \mathbf{x},\mathbf{y}\rangle|}{\|\mathbf{x}\|\|\mathbf{y}\|},\label{eq:theta1}
\end{align}
where $\theta_{min}^{[\mathcal{X},\mathcal{Y}]} = \theta_{min}^{[\mathcal{X}^{\perp},\mathcal{Y}^{\perp}]}$ \cite{Deutsch:1995tc,Szyld:2006bg}.
Then, note that 
\begin{align}
\|\widetilde{\Pi}\|^2 = \|\widetilde{\Pi}^*\|^2 = \csc^2\left(\theta_{min}^{[\mathcal{R}(\widetilde{\Pi}^*),\textnormal{ker}(\widetilde{\Pi}^*)]}\right)  =
	1 + \cot^2\left(\theta_{min}^{[\mathcal{R}(\widetilde{\Pi}^*),\textnormal{ker}(\widetilde{\Pi}^*)]}\right),\label{eq:cot}
\end{align}
where $\cot^2\left(\theta_{min}^{[\mathcal{R}(\widetilde{\Pi}^*),\textnormal{ker}(\widetilde{\Pi}^*)]}\right)$ accounts
for the non-orthogonal part of $\widetilde{\Pi}$ \cite{Szyld:2006bg}. Expanding residual propagation of coarse-grid correction yields
\begin{align*}
\widetilde{\Pi} & = A\Pi A^{-1} = AP(RAP)^{-1}R = \begin{pmatrix} \mathbf{0} & (I - A_{ff}\Delta)A_{fc}\mathcal{K}^{-1} \\ \mathbf{0} & I\end{pmatrix}.
\end{align*}
Note that $\mathcal{R}(\widetilde{\Pi}^*) = (\mathbf{0},\mathbf{v}_c^T)^T$, for all $\mathbf{v}_c$ and ker$(\widetilde{\Pi}^*)$
is given by block vectors with the form $\mathbf{x}_c = -\mathcal{K}^{-*}A_{fc}^*(I - A_{ff} \Delta)^* \mathbf{x}_f$. Plugging in
to \eqref{eq:theta1} yields
\begin{align*}
\cos\left(\theta_{min}^{[\mathcal{R}(\widetilde{\Pi}^*),\textnormal{ker}(\widetilde{\Pi}^*)]}\right) & = 
	\sup_{\substack{\mathbf{v}_c\neq\mathbf{0},\\ \mathbf{x}_f \neq\mathbf{0}} }
	\frac{\Big| \Big\langle \mathbf{v}_c, \mathcal{K}^{-*}A_{fc}^*(I - A_{ff}\Delta)^*\mathbf{x}_f\Big\rangle\Big| }{\|\mathbf{v}_c\|\sqrt{\|\mathbf{x}_f\|^2 +
	\|\mathcal{K}^{-*}A_{fc}^*(I - A_{ff}\Delta)^*\mathbf{x}_f\|^2}} \\
& = \sup_{\mathbf{x}_f \neq\mathbf{0}} \frac{\|\mathcal{K}^{-*}A_{fc}^*(I - A_{ff}\Delta)^*\mathbf{x}_f\| }{\sqrt{\|\mathbf{x}_f\|^2 +
	\|\mathcal{K}^{-*}A_{fc}^*(I - A_{ff}\Delta)^*\mathbf{x}_f\|^2}}, \\
\cos^2\left(\theta_{min}^{[\mathcal{R}(\Pi),\textnormal{ker}(\Pi)]}\right) & = \sup_{\mathbf{x}_f \neq\mathbf{0}}
	\frac{\|\mathcal{K}^{-*}A_{fc}^*(I - A_{ff}\Delta)^*\mathbf{x}_f\|^2 }{\|\mathbf{x}_f\|^2 +
	\|\mathcal{K}^{-*}A_{fc}^*(I - A_{ff}\Delta)^*\mathbf{x}_f\|^2} \\
& = \sup_k \frac{\sigma_k}{1 + \sigma_k},
\end{align*}
where $\{\sigma_k\}_{k=1}^{n_f}$ are the singular values of $\mathcal{K}^{-*}A_{fc}^*(I - A_{ff}\Delta)^*$ and, equivalently,
$(I - A_{ff}\Delta)A_{fc}\mathcal{K}^{-1}$, in decreasing order. Then, $\theta_{min}^{[\mathcal{R}(\widetilde{\Pi}^*),
\textnormal{ker}(\widetilde{\Pi}^*)]} = \arccos\bigg(\sqrt{\frac{\sigma_1}{1+\sigma_1}}\bigg)$,
where $\sigma_1$ is the largest singular value of $(I - A_{ff}\Delta)A_{fc}\mathcal{K}^{-1}$.
Plugging in, $\cot\Big(\theta_{min}^{[\mathcal{R}(\widetilde{\Pi}^*),\textnormal{ker}(\widetilde{\Pi}^*)]}\Big) = \sqrt{\sigma_1}$,
and $\|\Pi\|^2 = 1 + \sigma_1$.
\end{proof}

Because the largest singular value of $\mathcal{K}^{-1}$ is expected to scale like $\frac{1}{h}$ for advection and $\frac{1}{h^2}$ for
diffusion, this indicates that for $Z = \mathbf{0}$, an increasingly accurate approximation of $R_{\textnormal{ideal}}$ must be
made to guarantee that the coarse-grid correction $\widetilde{\Pi}$, remains bounded as $h\to0$.

A theoretical understanding of ideal interpolation in two-grid convergence for SPD matrices and exactly how it is ideal can be found in
\cite{Falgout:2004cs}, along with a more recent generalization to a class of ideal interpolation operators in \cite{ideal}.
In general, ideal interpolation may result in a dense $W$, and is 
impractical to form. However, the goal in much of AMG literature (aside from choosing a good C/F splitting) is to build a sparse
approximation to $P_{\textnormal{ideal}}$. For F-relaxation to converge quickly, $A_{ff}$ needs to be well-conditioned, in which case a sparse
approximation to $A_{ff}^{-1}$ is possible \cite{Brannick:2007fb}.

\subsection{Ideal restriction and error reduction}\label{sec:background:ir}

The focus in AMG (and most multigrid methods) has long been the accuracy of interpolation. The above seems to suggest
that the choice of restriction is less important than interpolation. For symmetric matrices, there are good reasons to take $R = P^T$:
the coarse-grid matrix remains symmetric and, moreover, coarse-grid correction is then an orthogonal projection in the $A$-norm,
which is typically what AMG convergence is measured in. In this setting, interpolation really is the determining factor in AMG convergence.
For nonsymmetric problems, $A$ no longer defines a valid norm, and the loss of orthogonality in coarse-grid correction is
difficult to avoid. In this case, the choice $R = P^T$ is somewhat arbitrary, although it can be effective for some problems
\cite{ruge:1987}. However, other choices are possible; for
example, see \cite{Lottes:2017jz,Manteuffel:2017,amgir,Olson:2011fg,Sala:2008cv,Wiesner:2014cy}. Here, we consider the fact that,
as with interpolation, there is, in a sense, an ``ideal'' restriction operator. Recall that ideal interpolation was based on the residual;
\textit{ideal restriction} is instead based on the error. 

One role of restriction can be seen as attenuating the effect of error components not in the range of interpolation.
Let $P$ be interpolation, as in \eqref{eq:PR}. Then, any error vector, $\mathbf{e}$, can be decomposed into two components: one that is
interpolated from the error at C-points, and the remainder:
\begin{align*}
\mathbf{e} = \begin{pmatrix} \mathbf{e}_f \\ \mathbf{e}_c \end{pmatrix} =
	 \begin{pmatrix} W\mathbf{e}_c \\ \mathbf{e}_c \end{pmatrix} + \begin{pmatrix} \delta\mathbf{e}_f \\ \mathbf{0} \end{pmatrix},
\end{align*}
where $\delta\mathbf{e}_f$ is the ``contamination'' of F-point error, which is not in the range of interpolation.
If $\delta\mathbf{e}_f = \mathbf{0}$, then the F-point error is in the range of interpolation, and coarse-grid correction is
exact (see Section \ref{sec:background:ii}).
Here, we focus on the effect of restriction on error-propagation of coarse-grid correction \eqref{eq:err_prop}:
\begin{align}
\mathbf{e}^{(i+1)} & = \begin{pmatrix}\mathbf{e}_f^{(i)} \\\mathbf{e}_c^{(i)} \end{pmatrix} - P(RAP)^{-1}RA\left[
	\begin{pmatrix}W\mathbf{e}_c^{(i)} \\ \mathbf{e}_c^{(i)} \end{pmatrix} + \begin{pmatrix}\delta\mathbf{e}_f^{(i)} 
	\\\mathbf{0}\end{pmatrix}\right] \nonumber\\
& = \begin{pmatrix}\mathbf{e}_f^{(i)} \\\mathbf{e}_c^{(i)} \end{pmatrix} - P(RAP)^{-1}RA\left[ P\mathbf{e}_c^{(i)} + 
	\begin{pmatrix}\delta\mathbf{e}_f^{(i)} \\\mathbf{0}\end{pmatrix}\right] \nonumber\\
& = \begin{pmatrix}\mathbf{e}_f^{(i)} - W\mathbf{e}_c^{(i)} \\\mathbf{0}\end{pmatrix} - P(RAP)^{-1}RA\begin{pmatrix}\delta\mathbf{e}_f^{(i)}\\
	\mathbf{0}\end{pmatrix} \label{eq:cpt} \\
& = \left[I - P(RAP)^{-1}RA\right]\begin{pmatrix}\delta\mathbf{e}_f^{(i)}\\ \mathbf{0}\end{pmatrix} . \nonumber
\end{align}
\tcr{Looking at \eqref{eq:cpt}, it would be desirable to choose operators such that the second term is zero for all $\delta\mathbf{e}_f$. In
that case, the solution after coarse-grid correction would have a correction at F-points defined by $W$,
$\mathbf{e}_f^{(i+1)} = \mathbf{e}_f^{(i)} - W\mathbf{e}_c^{(i)}$, and the solution at C-points would be exact. Coupling this
with a convergent F-relaxation scheme gives a convergent two-grid method. Setting the latter term in \eqref{eq:cpt}
equal to zero for all $\delta\mathbf{e}_f$ can also be seen as eliminating the contribution of $\delta\mathbf{e}_f$ to the coarse-grid
right-hand side, which is equivalent to setting 
$RA\begin{pmatrix}\delta\mathbf{e}_f\\\mathbf{0}\end{pmatrix} = \mathbf{0}$ for all $\delta\mathbf{e}_f$.} Expanding $RA$, we have
\begin{align}
\mathbf{0} = \begin{pmatrix} Z & I \end{pmatrix}\begin{pmatrix} A_{ff} & A_{fc} \\ A_{cf} & A_{cc}\end{pmatrix} \begin{pmatrix}\delta\mathbf{e}_f\\\mathbf{0}\end{pmatrix} & = (ZA_{ff} + A_{cf})\delta\mathbf{e}_f, \label{eq:def0}
\end{align}
which is satisfied by $Z= -A_{cf}A_{ff}^{-1}$. This leads to the \textit{ideal restriction} operator,
\begin{align}
R_{\textnormal{ideal}} & = \begin{pmatrix} -A_{cf}A_{ff}^{-1} & I \end{pmatrix}.\label{eq:r_ideal}
\end{align}
In fact, $R_{\textnormal{ideal}}$ is the
unique operator that gives an exact correction at C-points, independent of the interpolation operator \cite{amgir}.
Similar to the case of ideal interpolation, letting $W := \mathbf{0}$ gives an $\ell^2$-orthogonal coarse-grid
correction. \tcr{In practice, however, $R_{\textnormal{ideal}}$ cannot be obtained exactly, and an analagous result as Lemma
\ref{lem:pbad} indicates that $W$ must be chosen to compensate for this imprecision for a scalable algorithm.
Nevertheless, in the context of error reduction},
the accurate correction at C-points obtained with $R\approx R_{\textnormal{ideal}}$ is generally
best followed by F-relaxation to distribute the new accuracy at C-points to F-points.

\subsection{Block-Analysis of AMG}\label{sec:background:block}

Let $\Delta$ be some approximation to $A_{ff}^{-1}$ defining our F-relaxation scheme, and $W$ and $Z$ some
interpolation and restriction operators over F-points, respectively. Denote $\mathcal{K} := RAP$. Then error propagation
of a two-level scheme with post F-relaxation takes the form
\begin{align*}
\mathcal{E} & = \underbrace{ \begin{pmatrix} I - \Delta A_{ff} & -\Delta A_{fc} \\ 0 & I \end{pmatrix}}_{F-relaxation}
\underbrace{\begin{pmatrix} I - W\mathcal{K}^{-1}(ZA_{ff}+A_{cf}) & -W\mathcal{K}^{-1}(ZA_{fc}+A_{cc}) \\ 
	-\mathcal{K}^{-1}(ZA_{ff}+A_{cf}) & I - \mathcal{K}^{-1}(ZA_{fc}+A_{cc}) \end{pmatrix}}_{Coarse-grid\text{ }correction}\\
& = \begin{pmatrix} I - \Delta A_{ff} - \widehat{W}\mathcal{K}^{-1}(ZA_{ff}+A_{cf}) &
	-\Delta A_{ff} - \widehat{W}\mathcal{K}^{-1}(ZA_{fc}+A_{cc}) \\ 
	-\mathcal{K}^{-1}(ZA_{ff}+A_{cf}) & I - \mathcal{K}^{-1}(ZA_{fc}+A_{cc}) \end{pmatrix},
\end{align*}
where $\widehat{W} := (I - \Delta A_{ff})W - \Delta A_{fc}$ and $\mathcal{K} := RAP = ZA_{ff}W + ZA_{fc} + A_{cf}W + A_{cc}$.
We refer to $\widehat{W}$ as the \textit{effective interpolation}, because a little bit of algebra shows that we can expand
error-propagation to take the form of an approximate LDU preconditioner for $A$:
\begin{align*}
\mathcal{E} & = I - M^{-1}A \\
& = \begin{pmatrix} I & 0 \\ 0 & I \end{pmatrix} - \begin{pmatrix}I & \widehat{W} \\ 0 & I\end{pmatrix} 
	\begin{pmatrix} \Delta & 0 \\0 & \mathcal{K}^{-1}\end{pmatrix} \begin{pmatrix} I & 0 \\ Z & I\end{pmatrix}A.
\end{align*}
Here, the LDU preconditioner, $M$, can be collapsed:
\begin{align*}
M & = \begin{pmatrix} I & 0 \\ -Z & I\end{pmatrix} \begin{pmatrix} \Delta^{-1} & 0 \\0 & \mathcal{K}\end{pmatrix}
	\begin{pmatrix}I & -\widehat{W} \\ 0 & I\end{pmatrix} \\
& = \begin{pmatrix} \Delta^{-1} & A_{fc} + (A_{ff} - \Delta^{-1})W \\ -Z\Delta^{-1} & (Z\Delta^{-1}+A_{cf})W + A_{cc}\end{pmatrix}.
\end{align*}
In looking at convergence of preconditioners, it is useful to let $A = M+N$ and notice that $I - M^{-1}A = -M^{-1}N$. 
Here, $N$ reduces to the following outer product:
\begin{align*}
N & = \begin{pmatrix} A_{ff} - \Delta^{-1} & -(A_{ff} - \Delta^{-1})W \\ Z\Delta^{-1} + A_{cf} & -(A_{cf}+Z\Delta^{-1})W \end{pmatrix} \\
& = \begin{pmatrix} A_{ff} - \Delta^{-1} \\ Z\Delta^{-1} + A_{cf} \end{pmatrix} \begin{pmatrix} I & -W \end{pmatrix}.
\end{align*}
Then,
\begin{align*}
(I - M^{-1}A) & = -\begin{pmatrix}I & \widehat{W} \\ 0 & I\end{pmatrix}  \begin{pmatrix} \Delta & 0 \\0 & \mathcal{K}^{-1}\end{pmatrix}\
	\begin{pmatrix} I & 0 \\ Z & I\end{pmatrix} \begin{pmatrix} A_{ff} - \Delta^{-1} \\ Z\Delta^{-1} + A_{cf} \end{pmatrix}
	\begin{pmatrix} I & -W \end{pmatrix} \\
& = \begin{pmatrix} I - \Delta A_{ff} - \widehat{W}\mathcal{K}^{-1}(ZA_{ff} + A_{cf}) \\ -\mathcal{K}^{-1}(ZA_{ff} + A_{cf}) \end{pmatrix}
	\begin{pmatrix} I & -W \end{pmatrix} .
\end{align*}
Denoting $H := I - \Delta A_{ff} - \widehat{W}\mathcal{K}^{-1}(ZA_{ff} + A_{cf})$ and $J := -\mathcal{K}^{-1}(ZA_{ff} + A_{cf})$, then
the null space and range of $I - M^{-1}A$ are given respectively by
\begin{align*}
\mathcal{N}(I - M^{-1}A) &= \begin{pmatrix}W \\ I\end{pmatrix}\overline{\zeta} \hspace{3ex} \forall\text{ }\overline{\zeta}\in\mathbb{R}^{n_c}, \\
\mathcal{R}(I - M^{-1}A) &=  \begin{pmatrix}H \\ J\end{pmatrix}\overline{\eta}  \hspace{3ex} \forall\text{ }\overline{\eta}\in\mathbb{R}^{n_f}.
\end{align*}
Note that the null space is just the range of interpolation because we have assumed an exact coarse-grid solve. Eigenvectors of
$I - M^{-1}A$ with nonzero eigenvalues must then take the form $\begin{pmatrix}H \\ J\end{pmatrix}\eta$ for some
$\eta \in \mathbb{R}^{n_f}$. Then the eigenvalue problem $(I - M^{-1}A)\mathbf{v} = \lambda\mathbf{v}$ takes the form
\begin{align*}
\begin{pmatrix}H \\ J\end{pmatrix}\begin{pmatrix} I & -W \end{pmatrix} \mathbf{v} & =  \lambda\mathbf{v}, \\
\begin{pmatrix}H \\ J\end{pmatrix}\begin{pmatrix} I & -W \end{pmatrix} \begin{pmatrix}H \\ J\end{pmatrix}\eta & =  \lambda\begin{pmatrix}H \\ J\end{pmatrix}\eta, \\
\begin{pmatrix}H \\ J\end{pmatrix} (H - WJ)\eta & = \begin{pmatrix}H \\ J\end{pmatrix} \lambda\eta.
\end{align*}
Thus, eigenvalues of $I - M^{-1}A$ corresponding to the range are given by eigenvalues of
\begin{align}
G :&= H - WJ \nonumber\\
& = I - \Delta A_{ff} - \widehat{W}\mathcal{K}^{-1}(ZA_{ff} + A_{cf}) + W\mathcal{K}^{-1}(ZA_{ff} + A_{cf}) \nonumber\\
& = I - \Delta A_{ff} - (\widehat{W} - W)\mathcal{K}^{-1}(ZA_{ff} + A_{cf}) \nonumber\\
& = (I - \Delta A_{ff}) + \Delta(A_{ff}W + A_{fc})\mathcal{K}^{-1}(ZA_{ff} + A_{cf}) \label{eq:2grid_spec}.
\end{align}
Terms $(A_{ff}W+A_{fc})$ and $(ZA_{ff}+A_{cf})$ represent how close to ideal our given interpolation and restriction operators
are. If $E_{TG}$ is the two-grid error-propagation operator, then if $W$ \textit{or} $Z$ are ideal, $\rho(E_{TG}) = 1 - \Delta A_{ff}$. 
On the other hand, for an exact solve on F-points ($\Delta := A_{ff}^{-1}$) and general $W$ and $Z$,
\begin{align}
\rho(E_{TG}) & = \rho\Big( (W + A_{ff}^{-1}A_{fc})\mathcal{K}^{-1}(ZA_{ff} + A_{cf})\Big). \label{eq:cgc_rho}
\end{align}
An exact solve on F-points coupled with ideal interpolation \textit{or} ideal restriction results in $\rho(E_{TG})=0$.

Looking at \eqref{eq:2grid_spec}, the first term corresponds to F-relaxation and the second to coarse-grid correction. For
SPD matrices, it is natural to make the coarse-grid correction term small when relaxation error is large, and vice versa. 
A similar result holds in the nonsymmetric setting. In \cite{amgir}, an expanded analysis develops sufficient conditions on $R$, $P$,
and $\Delta$ to bound $\|G\| < 1$. In particular, $P$ must satisfy a classical multigrid approximation property, with the strength of
the approximation property (in terms of how accurately low-energy modes are interpolated) governed by how accurately $R$
approximates $R_{\textnormal{ideal}}$ and $\Delta$ approximates $A_{ff}^{-1}$. Less accurate approximations to ideal restriction and $A_{ff}^{-1}$
require that $P$ satisfy a stronger approximation property, while more accurate approximations make the requirement on $P$ weaker. 

In the case of block-triangular matrices resulting from the discretization of advection, sparse and highly accurate approximations to
$A_{ff}^{-1}$ and $R_{\textnormal{ideal}}$ can often be formed \cite{amgir}. Here, a local approximation to the ideal restriction operator ($\ell$AIR)
is developed, which provides an accurate approximation to the action of $R_{\textnormal{ideal}}$, even in the case of diffusive-like matrices. However,
for diffusion-dominated problems, we cannot rely on as accurate of approximations to $R_{\textnormal{ideal}}$ as in the purely advective case. 
Fortunately, building interpolation operators to capture low-energy modes in the case of diffusive-like problems (and thereby satisfy
approximation properties) is a well-studied topic. Building on the analysis here and in \cite{amgir}, classical multigrid interpolation
formulae are coupled with the newly developed $\ell$AIR algorithm for a robust nonsymmetric solver. 

\section{Local approximate ideal restriction ($\ell$AIR)}\label{sec:air}

\subsection{Local approximate ideal restriction}\label{sec:air:air}

As shown in Section \ref{sec:background:ir}, ideal restriction can be motivated through eliminating the contribution of error at
F-points to the coarse-grid right-hand side. Thus, to build $R$, we try to do this ``locally'' for each C-point (corresponding to a given
row of $R$). For each $i$th C-point, a restriction neighborhood $\mathcal{R}_i$ (consisting of some ``nearby'' F-points) is chosen,
and the idea is to choose restriction weights $z_{ik}$ for each $k\in\mathcal{R}_i$ so that the effect of perturbing error at
any $j\in\mathcal{R}_i$ on the residual at $i$ is zero. Note that a unit change in error at point $j$ changes
the residual at point $i$ by $a_{ij}$ and the residual at each point $k$ in $\mathcal{R}_i$ by $a_{kj}$. Requiring
that restriction weights be defined so that the total effect of these changes on the residual at point $i$ is zero
gives the following equation:
\begin{align}\label{eq:johnair}
a_{ij} + \sum_{k\in\mathcal{R}_i} z_{ik}a_{kj} = 0.
\end{align}
Solving \eqref{eq:johnair} for all $j\in\mathcal{R}_i$ determines the $i$th row of $Z$, where
$R = (Z, I)$, and is equivalent to setting $(RA)_{ik} = 0$ for all $k$ such that
$(i,k)\in\mathcal{R}_i$. This is simply setting $RA$ equal to zero within a pre-determined F-point sparsity pattern for
$R$. Note that this can also be seen as directly approximating the action of $R_{\textnormal{ideal}}$ on F-points, where
$R_{\textnormal{ideal}}A = (\mathbf{0} , S_A)$, for Schur complement $S_A$. In either case, denoting indices of the sparsity
pattern for the $i$th row of $R$ as $\mathcal{R}_i = \{\ell_1,...,\ell_{S_i}\}$, where
$S_i = |\mathcal{R}_i|$ is the size of the sparsity pattern, the resulting linear system takes the form
\begin{align}
\begin{pmatrix} a_{\ell_0 \ell_0} & a_{\ell_1\ell_0} & ... & a_{\ell_{S_i}\ell_0} \\
a_{\ell_0\ell_1} & a_{\ell_1\ell_1} & ... & a_{\ell_{S_i}\ell_1} \\
\vdots & & \ddots & \vdots \\
a_{\ell_0\ell_{S_i}} & a_{\ell_1\ell_{S_i}} & ... & a_{\ell_{S_i}\ell_{S_i}}\end{pmatrix}
\begin{pmatrix} z_{i\ell_0} \\ z_{i \ell_1} \\ \vdots \\ z_{i\ell_{S_i}}\end{pmatrix}
& = - \begin{pmatrix} a_{i\ell_0} \\ a_{i\ell_1} \\ \vdots \\ a_{i\ell_{S_i}}\end{pmatrix} \label{eq:system}.
\end{align}
For matrices with a $k\times k$ block structure, an equivalent system to \eqref{eq:system} can be formed
based on block connections, where $a_{\ell_i,\ell_j}$ is a $k\times k$ block in the matrix $A$ and $z_{i,\ell_j}$
a $k\times k$ block in $R$. \tcr{Such a block structure may arise from a discretized set of $k$ equations with $k$
unknowns, or in certain discontinuous discretizations, such as discontinuous Galerkin, where degrees-of-freedom
for an individual finite element correspond to a non-overlapping block in the matrix.}
Solving \eqref{eq:system} for $k$ right-hand sides determines all elements for
$R$ in block form. If $A_{ff}$ is diagonally dominant, which should be the case given an appropriate CF-splitting,
then \eqref{eq:system} is nonsingular and has a unique solution.\footnote{\tcr{There have been situations on coarse
levels in the hierarchy where a local linear system is singular, and \eqref{eq:system} is formulated as a least-squares
problem to pick the minimal norm solution. Such systems are rare and the cause has been difficult to isolate. For
example, singular or nearly singular local systems may arise from a nearly singular coarse-grid matrix, which is due
to a poor coarsening on the finer level, as opposed to the local choice of neighborhood for $R$. Understanding when
such systems arise and appropriate modifications to handle them, particularly in the nearly singular case, is ongoing work.}}
We refer to the proposed method for building $R$ as \textit{local approximate ideal restriction} ($\ell$AIR).

\tcr{Consistent with many AMG methods, the neighborhood of F-points for a given C-point is chosen as some set of
``strong'' connections, typically of graph distance one or two. One unique aspect of $\ell$AIR is that considering a larger
set of strong connections always leads to a more accurate approximation of $R_{\textnormal{ideal}}$. However,
for many problems, considering larger distance neighborhoods
can lead to intractable complexity due to coarse-grid operator fill in, as well as the $O(\ell^3)$ cost of solving dense linear
systems of size $\ell\times \ell$ for each row. For discretizations with significant diffusion, the number of neighbors of distance
$k$ scales roughly like $O(2dk^d)$, for dimension $d$, leading to a rapid increase in setup cost and operator complexity. 
On coarser levels in the hierarchy with increased matrix connectivity, the local system sizes can reach into the hundreds
for $k = 2$, which can be computationally expensive when solving for every row.
It is worth noting that for upwind discretizations of pure advection, the number of neighbors of distance
$k$ only increases like $O(k)$. Because of this, distance-three or -four restriction neighborhoods may be tractable 
for advective problems, which was considered in \cite{amgir}. However, even on strictly advective problems, the additional
operator complexity of using long-distance neighborhoods for restriction was not justified by improved
convergence \cite{amgir}, and solving the larger dense linear systems here resulting from long-distance neighborhoods
also adds significant setup cost. For these
reasons, we limit numerical results to consider distance-one and -two neighborhoods for restriction.
}

\begin{remark}[Row scaling]

In some cases, the row scaling of a matrix can cause problems for classical AMG interpolation formulae, leading to negative
diagonal entries in the coarse-grid operator. Although such discretizations are not commonplace, it is worth pointing out that
$\ell$AIR is insensitive to row scaling. Suppose that the fine-grid matrix is scaled by some diagonal matrix, $D$: $\tilde{A} := DA$.
Then let $\tilde{R}$ and $\tilde{Z}$ denote the corresponding local approximate ideal restriction operator and its F-block, respectively. 
Weights $\tilde{z}_{ik}$ are given by solving 
\begin{align*}
\tilde{a}_{ij} + \sum_{k\in\mathcal{R}_i} \tilde{z}_{ik}\tilde{a}_{kj} = d_ia_{ij} + \sum_{k\in\mathcal{R}_i} \tilde{z}_{ik}d_ka_{kj} = 0,
\end{align*}
for all $(i,k)\in\mathcal{R}_i$, $i=0,...,n_c-1$. This is satisfied by $\tilde{z}_{ik} := d_iz_{ik}d_k^{-1}$, where $z_{ik}$
are the weights for $A$ satisfying \eqref{eq:johnair} and \eqref{eq:system}. It follows that $\tilde{Z} := D_cZD_f^{-1}$ and
$\tilde{R} := D_cRD^{-1}$. The resulting coarse-grid operator is then defined as 
\begin{align*}
\tilde{R}\tilde{A}P = D_cRD^{-1}DAP = D_cRAP,
\end{align*}
which is simply maintaining the fine-grid row scaling in the coarse-grid operator. In fact, looking at the error-propagation
operator,
\begin{align*}
I - P(\tilde{R}\tilde{A}P)^{-1}\tilde{R}\tilde{A} & = I - P(D_cRAP)^{-1}D_cRA \\
& = I - P(RAP)^{-1}RA,
\end{align*}
we see that error propagation of two-grid coarse-grid correction is independent of row-scaling. The same idea can be applied in a
recursive manner for a full multilevel hierarchy. This is a subtle feature of $\ell$AIR that makes it robust
for a wide class of problems and discretizations.
\end{remark}

\subsection{Comparison with Neumann series}\label{sec:air:neumann}

In \cite{amgir}, ideal restriction was approximated in the context of an AMG algorithm targeting triangular or block-triangular matrices.
For ease of notation, assume that $A$ is lower triangular with unit diagonal. 
Then $A_{ff} = I - L_{ff}$, for strictly lower triangular matrix $L_{ff}$, and $A_{ff}^{-1}$ can be written as a finite Neumann expansion:
\begin{align*}
A_{ff}^{-1} = \sum_{i=0}^{d_f+1} L_{ff}^i,
\end{align*}
where $d_f$ is the diameter of the graph of $A_{ff}$. 
Ideal restriction in \cite{amgir} is then approximated using a $k$th order Neumann approximation for some $k \ll n_f$:
$R = \begin{pmatrix} -A_{cf}\sum_{i=0}^k L_{ff}^i & I \end{pmatrix}.$
\tcr{The error in a truncated Neumann expansion of degree $k$ is given by
\begin{align*}
I - A_{ff}\sum_{i=0}^k L_{ff} = I - \left(\sum_{i=0}^k L_{ff}\right)A_{ff} = L_{ff}^{k+1}.
\end{align*}
In a graph-theoretic sense, $(L_{ff}^{k+1})_{ij}$ is the sum of all weighted walks of
length $k+1$ from node $i$ to node $j$, where a weight is given by the product of the edges in the walk. As discussed in \cite{amgir},
this provides an intuitive understanding of why a Neumann expansion is a good choice for triangular matrices. Consider an upwind
discretization of advection in three dimensions. Then:
\begin{enumerate}
\item Each node $i$ only has neighbors upstream in the direction of the flow.
\item Because information only flows one direction, if there is a path of length $k$ from $i$ to $j$, there is unlikely to be
a path of, say, length $k+2$. That is, as $k$ increases, the number of walks of length $k$ is unlikely to increase significantly
because there are no cycles in the graph. 
\item Because off-diagonal elements of differential operators are typically less than one, as $k$ increases, we expect the weight
of walks to decrease quickly. 
\end{enumerate}
When diffusion is introduced, points (1) and (2) no longer hold. Neighbors can now be found by reaching out in every direction, and
because information is flowing in every direction, the number of walks of length $k$ is increasing roughly exponentially with $k$. This
leads to a relatively dense matrix with large elements, and explains why a truncated Neumann approximation is a poor approximate
inverse for general matrices.}

\tcr{For matrices with triangular structure, the $k$th order Neumann approximation is exact for nodes with a \textit{maximum} distance
$\leq k$ from the diagonal. Equivalently, this eliminates the contribution of these nodes to the coarse-grid right-hand side. Here, we
recognize eliminating the contribution of F-point error to the coarse-grid right-hand side as an important function of the
restriction operator, and generalize the approximation developed in \cite{amgir}. $\ell$AIR as introduced here eliminates the contribution
of \textit{all} F-point error within distance $k$ (not only maximum distance $k$) to the coarse-grid right-hand side.
For some triangular matrices, these two approaches are equivalent, and, for discretizations of advection,
they are typically equivalent or similar. However, in general, the two approaches are quite different.}

To demonstrate the effectiveness of $\ell$AIR over the Neumann approximation,
consider two model problems: (i) the steady-state transport equation considered in \cite{amgir}, and (ii) a streamline upwind Petrov-Galerkin
(SUPG) discretization of an advection-dominated recirculating flow. Steady-state transport is constructed to be triangular in some ordering,
while recirculating flow has small symmetric diffusion components, and advection is not triangular due to the recirculating
velocity field. Table \ref{tab:neumann} gives a relative measure of how well we approximate the ideal operator using distance-one
and -two Neumann and $\ell$AIR for each problem. As expected, for the transport equation, $\ell$AIR and
Neumann perform identically. However, in the recirculating case, Neumann is hardly a more accurate approximation to ideal restriction than
just letting $Z = \mathbf{0}$, while $\ell$AIR retains a similar accuracy of approximation to the triangular case.
\begin{table}[!h]
\centering
\begin{tabular}{| c | c | c c c c |}\Xhline{1.25pt}
 & Problem & Neumann$_1$ & Neumann$_2$ & $\ell$AIR$_1$ & $\ell$AIR$_2$ \\\hline
\multirow{2}{*}{$\frac{\|Z + A_{cf}A_{ff}^{-1}\|_F}{\|A_{cf}A_{ff}^{-1}\|_F}$} &
	 Transport & 0.46 & 0.11 & 0.46 & 0.11 \\
	  & Recirculating & 0.99 & 0.99 & 0.46 & 0.17 \\\Xhline{1.25pt}
\end{tabular} 
\caption{Relative Frobenius error in approximating the F-block of ideal restriction using distance-one
	and -two Neumann and $\ell$AIR approximations, for steady-state transport and a recirculating flow. Each problem has
	approximately 6000 DOFs, and classical AMG coarsening is used to form a CF-splitting. }
\label{tab:neumann}
\end{table}

Although $\ell$AIR is a generalization of the Neumann approximation, the latter is an important contribution conceptually,
as it provides insight into how $n$AIR is able to achieve strong convergence factors on difficult problems and
high-order finite elements \cite{amgir}. In the triangular setting, it is relatively well understood how
good convergence is obtained with $n$AIR \cite{amgir}. Notwithstanding the block analysis in Section \ref{sec:background:block}, a good
understanding of when and why $\ell$AIR performs well, in general, is a focus of current research. 

\subsection{Filtering and lumping}\label{sec:air:lump}

A simple but effective technique for complexity reduction used in \cite{amgir} involves eliminating relatively small entries
from each row in the matrix, on every level in the hierarchy. For the steady-state transport equation considered there, it was found that
entries could be eliminated rather greedily without causing a significant degradation in convergence; for example, for all test problems,
entries were eliminated in row $i$ that were smaller than $0.001\cdot\max_j |a_{ij}|$. It is known that such an approach is typically not
effective on diffusive matrices, prompting research into more advanced techniques for reducing the number of matrix nonzeros \cite{Bienz:2015ve,
Falgout:2014uz,Treister:2015cp}. Here, we use a technique similar to the elimination used in \cite{amgir}, but instead of actually
eliminating entries, we add them to the diagonal in order to preserve the row sum. The concept of collapsing entries to the
diagonal is one of the original ideas of AMG \cite{Brandt:1985um,ruge:1987}, and proves to be a more robust technique than simple elimination
when diffusion is introduced. 

As an example, elimination and lumping are used on the discontinuous Galerkin (DG) discretization of advection-diffusion-reaction considered
in Section \ref{sec:results:dgu}. Table \ref{tab:filter} shows the average convergence factor (CF) and so-called work-unit-per-digit-of-accuracy,
denoted WPD, for various complexity reduction strategies. The WPD is introduced in detail with numerical results in Section
\ref{sec:results}; for now, consider it as a linear measure of time to solution, that is, dropping from 20 WPD to 10 is a $2\times$ reduction in
solve time. 
{
\begin{table}[!h]
\renewcommand{\tabcolsep}{0.195cm}
\small
\centering
\begin{tabular}{| c | c c | c c | c c | c c | c c | c c | c c |}\Xhline{1.25pt}
 & \multicolumn{2}{c|}{None} & \multicolumn{6}{c|}{Lumping} & \multicolumn{6}{c|}{Elimination} \\\hline
$\kappa$/$\theta_D$ & \multicolumn{2}{c|}{0} & \multicolumn{2}{c|}{0.001} & \multicolumn{2}{c|}{0.01} & \multicolumn{2}{c|}{0.1} &
	\multicolumn{2}{c|}{0.001} & \multicolumn{2}{c|}{0.01} & \multicolumn{2}{c|}{0.1} \\
& WPD & CF & WPD & CF & WPD & CF & WPD & CF & WPD & CF & WPD & CF & WPD & CF \\\hline
$10^{-10}$ & 12.2 & 0.37 & 9.5 & 0.38 & 9.1 & 0.38 & 11.3 & 0.49 & 9.6 & 0.38 & 9.1 & 0.38 & 12.7 & 0.53 \\
$10^{-7}$ & 12.1 & 0.37 &9.3 & 0.37 & 8.5 & 0.36 & 11.8 & 0.51 & 9.7 & 0.38 & 8.8 & 0.37 & 13.2 & 0.55 \\
$10^{-4}$ & 20.1 & 0.51 & 17.7 & 0.51 & 18.0 & 0.56 & 114.7 & 0.93 & 17.9 & 0.52 & 27.3 & 0.68 & 267.9 & 0.97 \\
$10^{-1}$ & 38.0 & 0.56 & 32.8 & 0.57 & 42.1 & 0.71 &  \multicolumn{2}{c|}{DNC} & 94.6 & 0.82 & 223.7 & 0.94 &  \multicolumn{2}{c|}{DNC} \\
$10$ & 45.0 & 0.61 & 38.4 & 0.62 & 48.9 & 0.74 &  \multicolumn{2}{c|}{DNC} & 105.7 & 0.84 & 179.6 & 0.92 & \multicolumn{2}{c|}{DNC}\\\Xhline{1.25pt}
\end{tabular} 
\caption{Results for various complexity-reduction techniques applied to a DG discretization of advection-diffusion-reaction with diffusion
coefficient $\kappa$ and elimination/lumping tolerance $\theta_D$. For all cases
tested, lumping entries in row $i$ smaller than $0.001\cdot\max_j |a_{ij}|$ to the diagonal results in convergence factors approximately
equal to those achieved without lumping/elimination, while decreasing the WPD by $10-25\%$. DNC denotes that the iterations did not converge.}
\label{tab:filter}
\end{table}
}

A lumping tolerance of $0.001$ has proven to be an effective choice for all problems we have tested. Solver complexity is
typically reduced 10--50$\%$, and convergence factors increase a small amount or none. For all $\ell$AIR results presented, all matrices
in the hierarchy are lumped with tolerance $0.001$. Although this is not a fundamental part of the solver, it provides a nice reduction
in complexity.

\section{Numerical results}\label{sec:results}

The model problem considered is the advection-diffusion-reaction
equation, 
\begin{align}
-\nabla\cdot\kappa\nabla u + \beta(x,y)\cdot\nabla u + \sigma u & = f \hspace{3ex}\text{in }\Omega,\label{eq:ad_diff_re}
\end{align}
defined over a convex domain, $\Omega$, with Lipschitz continuous boundary, $\Gamma$. This allows us to focus on a broad
class of PDEs and the effects of nonsymmetry on the solver. Specifically, two cases of
\eqref{eq:ad_diff_re} are considered. Section \ref{sec:results:supg} considers a divergence-free recirculating flow with Dirichlet
boundary conditions, discretized using an SUPG method. This problem comes up fairly often, and is a good initial test; however, in the
hyperbolic limit of no diffusion, the recirculating flow problem is not well-posed in that limit. In this sense, the recirculating flow can be thought
of as a diffusion problem with added advection. In order to consider the hyperbolic limit (and most nonsymmetric case), Section
\ref{sec:results:dgu} then considers an
upwind discontinuous Galerkin (DGu) discretization of an advection-diffusion-reaction equation with Dirichlet inflow boundaries, Neumann outflow
boundaries, and a velocity field $\beta$ with no closed curves or stationary points \cite{Ayuso:2009it,Sun:2005hm}. In the hyperbolic limit
of no diffusion, this problem reduces to the steady-state transport equation, and can be thought of as steady-state transport with added diffusion.
Further details on the discretization and results can be found in Section \ref{sec:results:supg}. Discretizations are generated using the
Dolfin finite element package \cite{dolfin}.

All results presented here use AMG V-cycles as a preconditioner for GMRES.
Coarsening is done using a classical AMG CF-splitting, with no second pass \cite{ruge:1987}. A second pass is not used because it is
not well suited for parallel environments (for example, it is not used in Hypre \cite{Henson:2002vk}), and an algorithm amenable to parallelization
is one of the important features of $\ell$AIR. The lack of a second pass in coarsening to adjust C-points and F-points is accounted for through modified 
interpolation routines \cite{DeSterck:2006et} that are used here. Strong connections for coarsening and the interpolation and restriction
neighborhoods are determined using a classical strength of connection (SOC) based on a hard minimum:
\begin{align*}
\mathcal{N}_i = \Big\{ j \text{ $|$ } i\neq j, -a_{ij} \geq \theta\max_{k\neq i} |a_{ik}| \Big\},
\end{align*}
for some tolerance $\theta$, where $\mathcal{N}_i$ is the neighborhood of strong connections to node $i$. For degree-two
neighborhoods in $\ell$AIR, we only consider F-F-C connections, not F-C-C connections. This is consistent with the goal
of approximating $A_{cf}A_{ff}^{-1}$ in $R_{\textnormal{ideal}}$, which does not contain F-C-C connections.
For coarsening, we use $\theta_C = 0.4$. It is interesting to note that tests indicate $\ell$AIR and classical interpolation methods perform
well with similar values of $\theta$. Thus, strong connections for degree-one interpolation and restriction use $\theta_1 =0.1$, and
strong connections for degree-two interpolation and restriction use $\theta_2 = 0.2$. The larger $\theta$ for degree-two operators is
meant to limit fill-in of the sparsity pattern. 

All results based on $\ell$AIR use one iteration of F-F-C Jacobi relaxation following coarse-grid correction, corresponding to two iterations of 
F-relaxation followed by one iteration of C-relaxation. In \cite{amgir} and in theory, only F-relaxation is discussed. In fact, for almost all
problems, F-relaxation is sufficient and adding one iteration of C-relaxation does
not improve convergence. However, for a few cases, typically either higher-order finite elements and/or diffusion-dominated problems,
one iteration of C-relaxation is necessary for good convergence. Thus, in order to demonstrate the method as robust with minimal parameter
tuning, we use F-F-C relaxation for all results. Classical AMG does not perform well with just F-relaxation, requiring a global relaxation scheme
for optimal results. To ensure the algorithms are similar from a comparative perspective, classical AMG methods are run with one
pre- and post-relaxation sweep of weighted Jacobi. 

In results shown below, the methods used to build restriction and interpolation are denoted as a pair $(R_{\textnormal{build}}, P_{\textnormal{build}})$, respectively. $\ell$AIR$_1$ refers to degree-one $\ell$AIR
and $\ell$AIR$_2$ to degree-two $\ell$AIR, and likewise for AMG$_1$ and AMG$_2$, specifically referring to modified classical interpolation 
(Eq. (4.4) in \cite{DeSterck:2006et}) and Extended$+i$ interpolation (Eqs. (4.10--4.11) in \cite{DeSterck:2008fc}), respectively. One-point
interpolation introduced in \cite{amgir}, where each F-point is interpolated by value from its strongest C neighbor, is denoted 1P,
and letting $R:=P^T$ for a Galerkin coarse grid is denoted $P^T$. All problems are solved to a $10^{-12}$ relative residual tolerance
with zero right-hand side and random initial guess. The approximate spatial mesh size is denoted $h$. 

Finally, we introduce the \textit{work-unit-per-digit-of-accuray} (WPD) as an objective measure of the efficiency of an AMG solver. One
\textit{work unit} (WU) measures the cost in floating-point operations (FLOPs) to perform a sparse matrix-vector multiply with the original
matrix. Other operations in the AMG solve are then measured with respect to this metric, and the \textit{cycle-complexity} (CC) is defined
as the total number of WUs to perform a single AMG cycle, including relaxation, computing the residual, restricting to the coarse grid,
and interpolating a correction. The relative efficiency of an AMG solver is then based on the cost of a single iteration and the convergence
factor $\rho$. One such measure is the total number of WUs to decrease the residual by one order of magnitude, that is, the
work-unit-per-digit-of-accuracy, which is given as
\begin{align*}
\chi_{wpd} := -\frac{CC}{\log_{10}\rho}.
\end{align*}
This measure is used to compare the performance of different solvers discussed here. Note that this does not include FLOPs
required in the setup phase. Although the setup phase can be a useful measure at times, here it is a poor reflection of performance
with respect to changes in problem size, or moving to a parallel environment. In particular, $\ell$AIR requires a small dense solve to build each
row of $R$, with size given by the local size of the restriction neighborhood (see \eqref{eq:system}), and cost which scales cubicly with size.
However, this system does not increase with problem size, and
is also strictly local and, thus, well suited for a parallel environment. A parallel implementation in the \textit{hypre} \cite{Falgout:2002vu}
library and a parallel performance model are ongoing work, but outside the scope of this paper.

\subsection{SUPG and Recirculating flow}\label{sec:results:supg}

One of the most popular discretizations for flow-problems is a stabilized upwind discretization, in particular, the \textit{streamline
upwind Petrov-Galerkin} (SUPG) finite element discretization, where artificial numerical diffusion is added in the direction of the velocity
field for stabilization purposes \cite{Brooks:1982bl}. Our first test problem is a two-dimensional recirculating flow discretized
with an SUPG discretization on a random, triangular, unstructured mesh. The continuous problem on domain $\Omega$ with boundary
$\Gamma$, diffusion coefficient $\kappa$, and velocity field $\beta$ is given as
\begin{align*}
-\nabla\cdot\kappa\nabla u + \beta\cdot\nabla u & = f\hspace{3ex} \text{ in }\Omega, \\
u & = 0 \hspace{3ex} \text{ on }\Gamma_0 \\
u & = 1 \hspace{3ex} \text{ on }\Gamma_1,
\end{align*}
where $\Omega = (0,1)\times(0,1)$, $\Gamma_1 = \{(x,y)\text{ $|$ } x=1,y\in[0,1]\text{ or }y=1,x\in[0,1]\}$, and
$\Gamma_0\cup\Gamma_1 = \Gamma$. The velocity field is a divergence-free recirculating flow given by
\begin{align*}
\beta(x,y) = \Big( x(1-x)(2y-1),  -(2x-1)(1-y)y \Big).
\end{align*}
The solution for varying levels of diffusion, $\kappa$, is shown in Figure \ref{fig:supg_adv_diff}.
\begin{figure}[!h]
  \centering
  \begin{subfigure}[b]{0.3\textwidth}
\includegraphics[width=\textwidth]{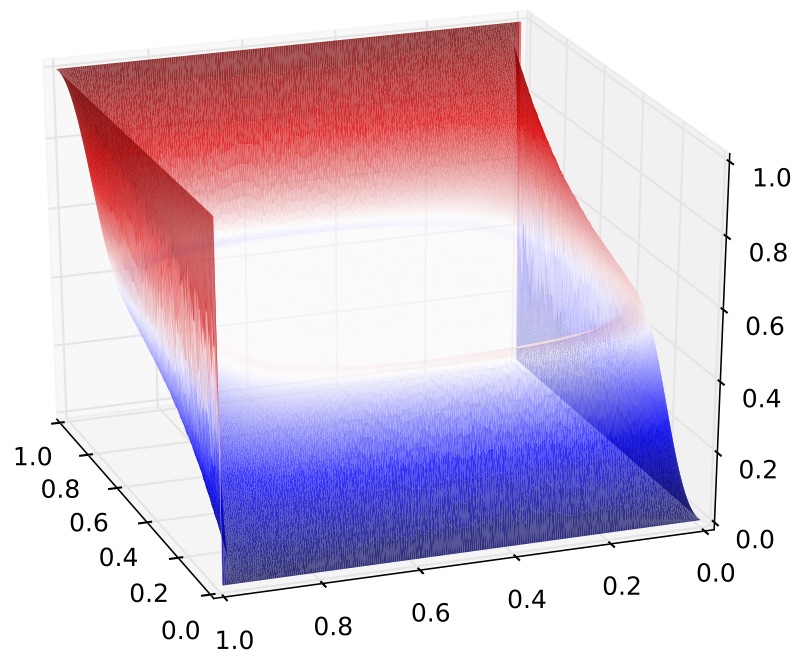}
    \caption{$\kappa = 10^{-4}$}
  \end{subfigure}
  \begin{subfigure}[b]{0.3\textwidth}
\includegraphics[width=\textwidth]{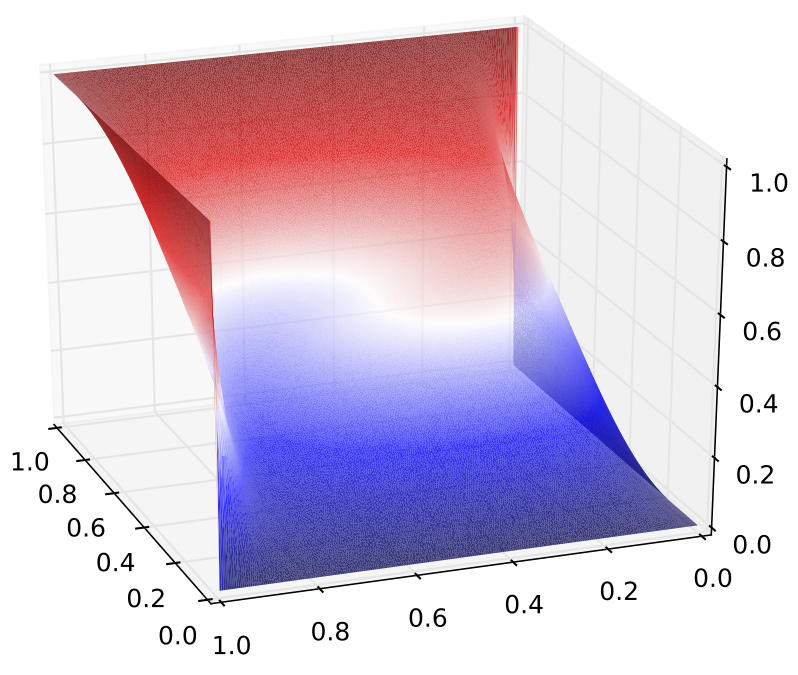}
    \caption{$\kappa = h = 0.005$}
  \end{subfigure}
  \begin{subfigure}[b]{0.3\textwidth}
\includegraphics[width=\textwidth]{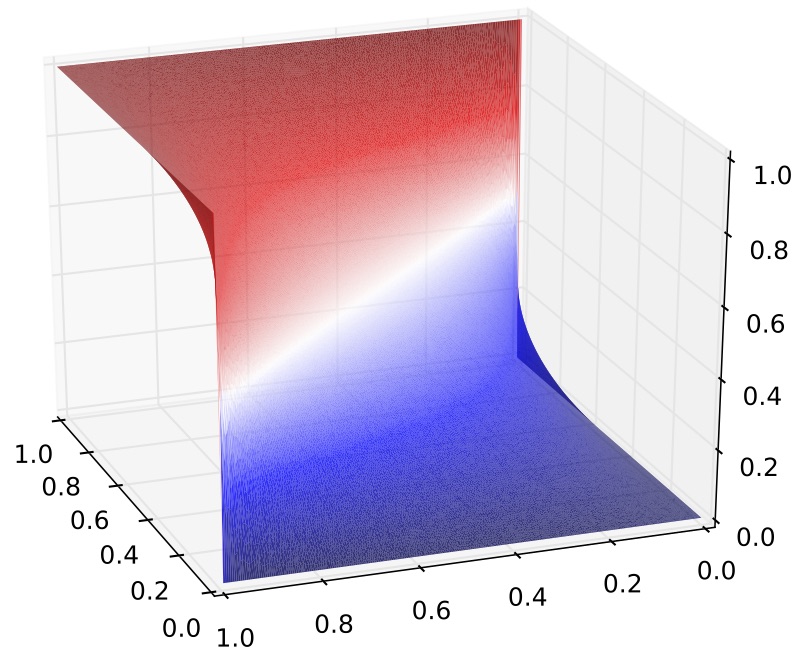}
    \caption{$\kappa = 1$}
  \end{subfigure}
  \caption{Solution of SUPG discretization of recirculating flow with varying
  		diffusion coefficients, $\kappa \in\{10^{-4},h,1\}$, representing the advection-dominated, equal advection
		and diffusion, and diffusion-dominated cases, respectively. }
\label{fig:supg_adv_diff}
\end{figure}

Here, results of $\ell$AIR and classical AMG applied to the recirculating flow are presented. In this case, as $\kappa \to 0$, the
problem is not well-posed and, in particular, the matrix becomes singular. Because of this, results are not expected to
be good for $\kappa \approx 0$. As an example, Table \ref{tab:supg_cond} shows the approximate condition number of
the matrices $A$ and $A_{ff}$ for $\kappa\in[10^{-10},1]$ and $h = \frac{1}{70}$. A pure diffusion discretization should have
an approximate condition number of $\frac{1}{h^2}$, while a pure advection discretization should have a condition number of approximately
$\frac{1}{h}$. For $\kappa = 1$, the conditioning of the SUPG discretization is close to the expected $\frac{1}{h^2}$, but as
$\kappa \to 0$, cond$(A) \to \infty$ as opposed to $\frac{1}{h}$, indicating a singular matrix.

\begin{table}[!h]
\centering
\begin{tabular}{| c | c c c c c c c |}\Xhline{1.25pt}
$\kappa$ & $10^{-10} $& $10^{-7}$ & $10^{-5}$ & 0.001 &  0.01 &  0.1 & 1 \\\hline
cond$(A)$ & $7.1\cdot10^9$ &  $4.5\cdot10^9$ & $1.23\cdot10^8$ & $1.23\cdot10^6$ & 123,300 & 12,332 & 2,279 \\
cond$(A_{ff})$ & 197,562 & 194,147 & 73,183 & 2,962 & 331 & 36 & 6.7 \\\Xhline{1.25pt}
\end{tabular} 
\caption{Condition number of $A$ and $A_{ff}$ as a function of diffusion coefficient $\kappa$ for degree-one finite elements
	on an unstructured mesh, with $h\approx \frac{1}{70}$ and 6300 DOFs (computed using NumPy \cite{numpy}).}
\label{tab:supg_cond}
\end{table}

$\ell$AIR performs well for reasonable values of $\kappa$. For linear finite elements (Figure \ref{fig:supg_kappa}),
$\ell$AIR slightly outperforms AMG for all $\kappa$ in terms of WPD and, moreover, is able to effectively solve the problem for
diffusion coefficients 1--2 orders of magnitude smaller than AMG, corresponding to matrix condition numbers likely two orders
of magnitude larger (Table \ref{tab:supg_cond}). Although results demonstrate $\ell$AIR as a robust solver and an improvement
over existing methods, we are not able to explore the highly nonsymmetric setting for which $\ell$AIR is designed because the
problem is not well-posed. This leads us to consider a time-dependent recirculating flow, where the linear system associated with
pure advection is well posed (Section \ref{sec:results:supg:time}), followed by a different variation of steady-state
advection-diffusion-reaction that is well-posed in the hyperbolic limit (Section \ref{sec:results:dgu}). 

\begin{figure}[!h]
  \centering
\captionsetup[subfigure]{justification=centering}
  \begin{subfigure}[b]{0.45\textwidth}
\includegraphics[width=\textwidth]{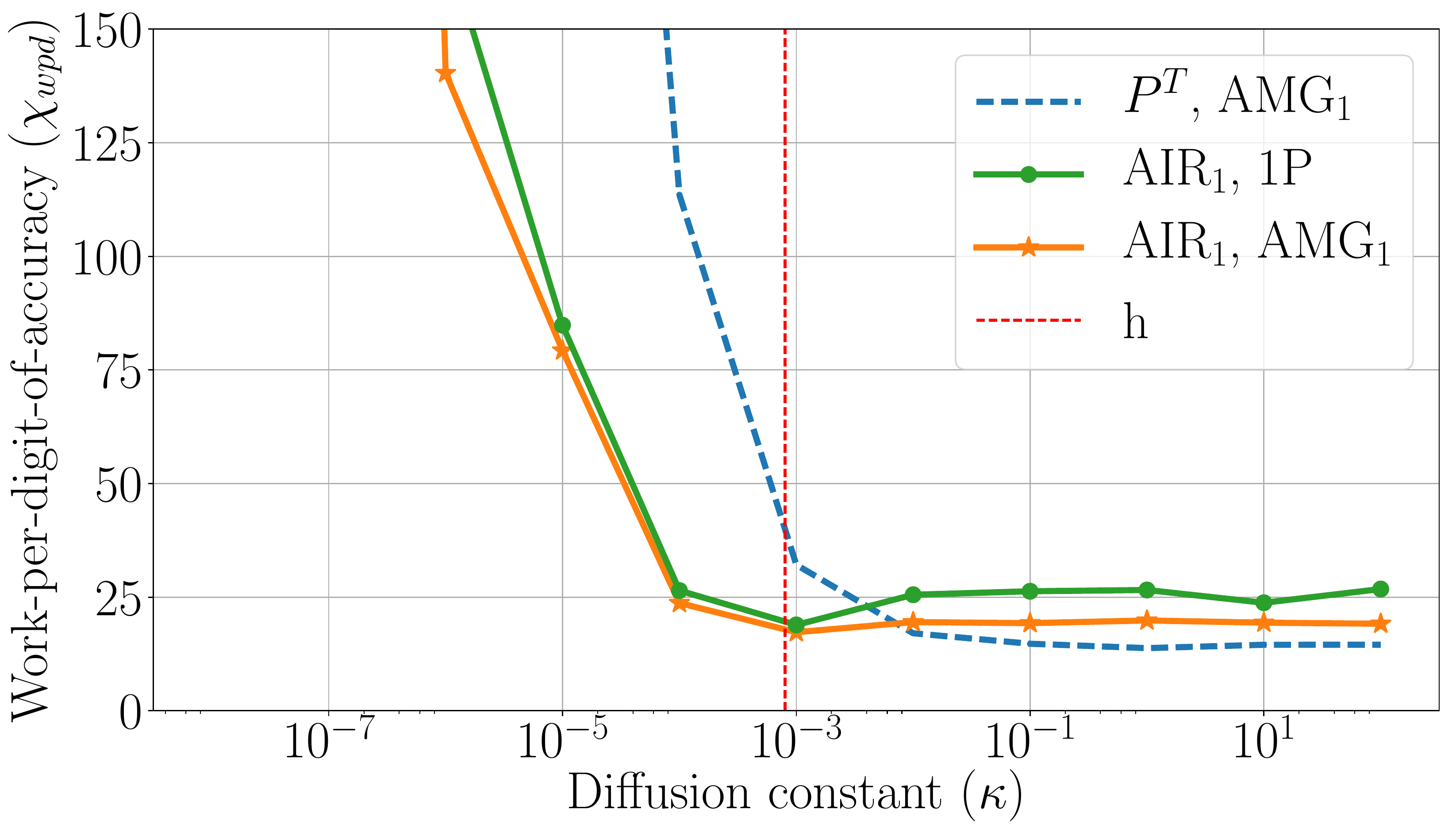}
    \caption{Degree-one finite element,\\degree-one interpolation/restriction}
  \end{subfigure}
  \begin{subfigure}[b]{0.45\textwidth}
\includegraphics[width=\textwidth]{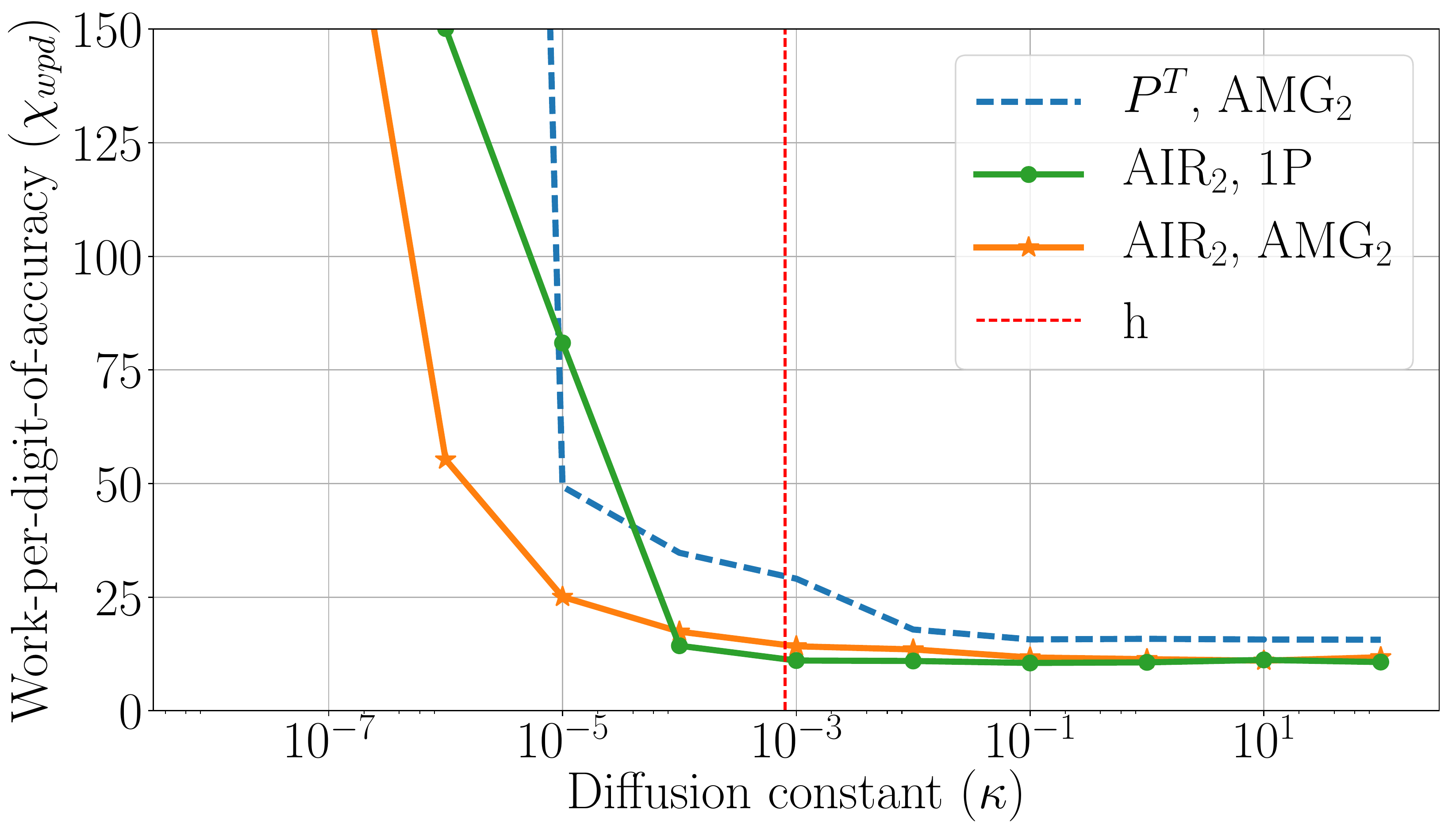}
    \caption{Degree-one finite element,\\degree-two interpolation/restriction}
  \end{subfigure}
     \begin{subfigure}[b]{0.45\textwidth}
\includegraphics[width=\textwidth]{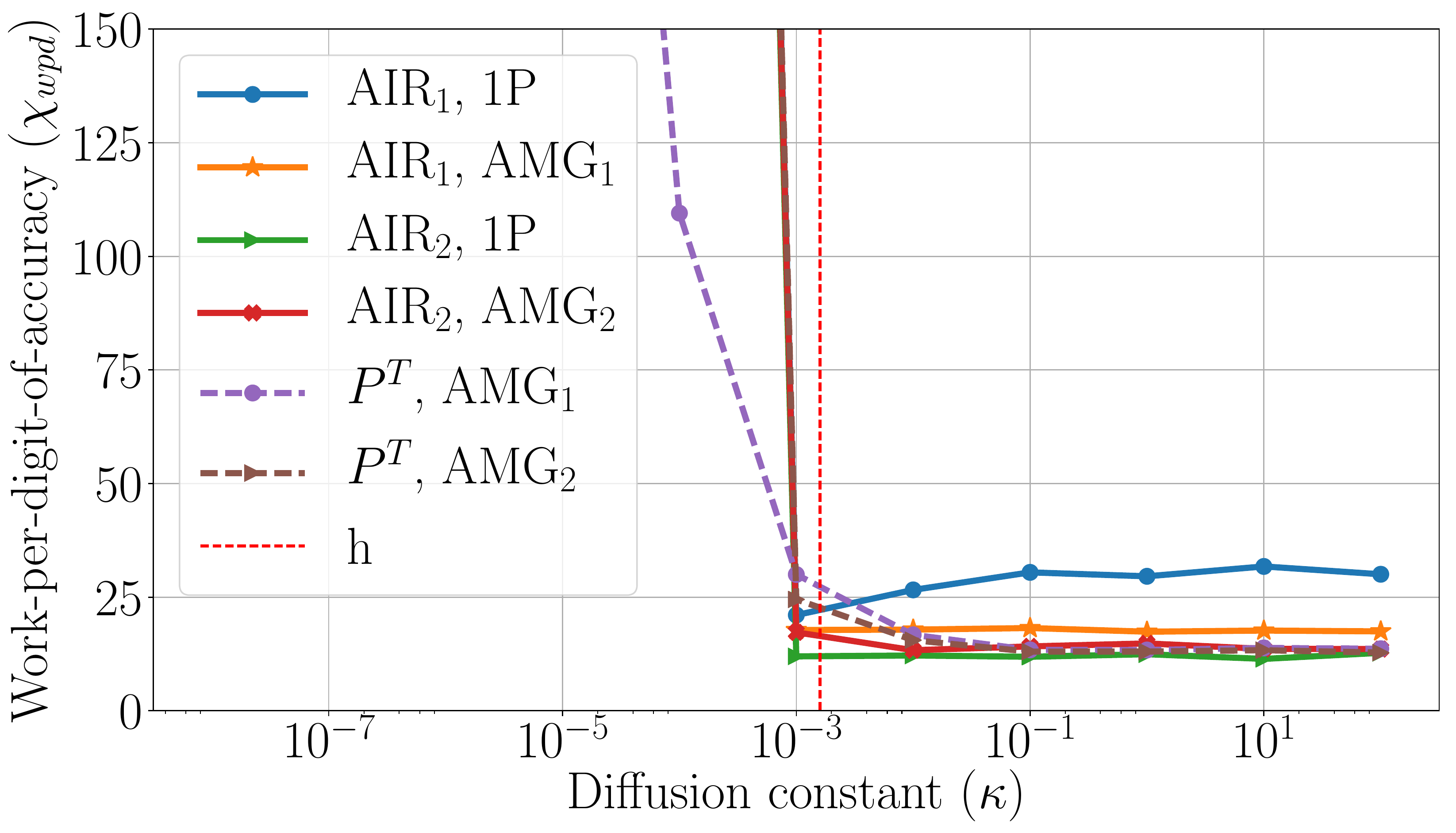}
    \caption{Degree-two finite elements}
  \end{subfigure}
  \begin{subfigure}[b]{0.45\textwidth}
\includegraphics[width=\textwidth]{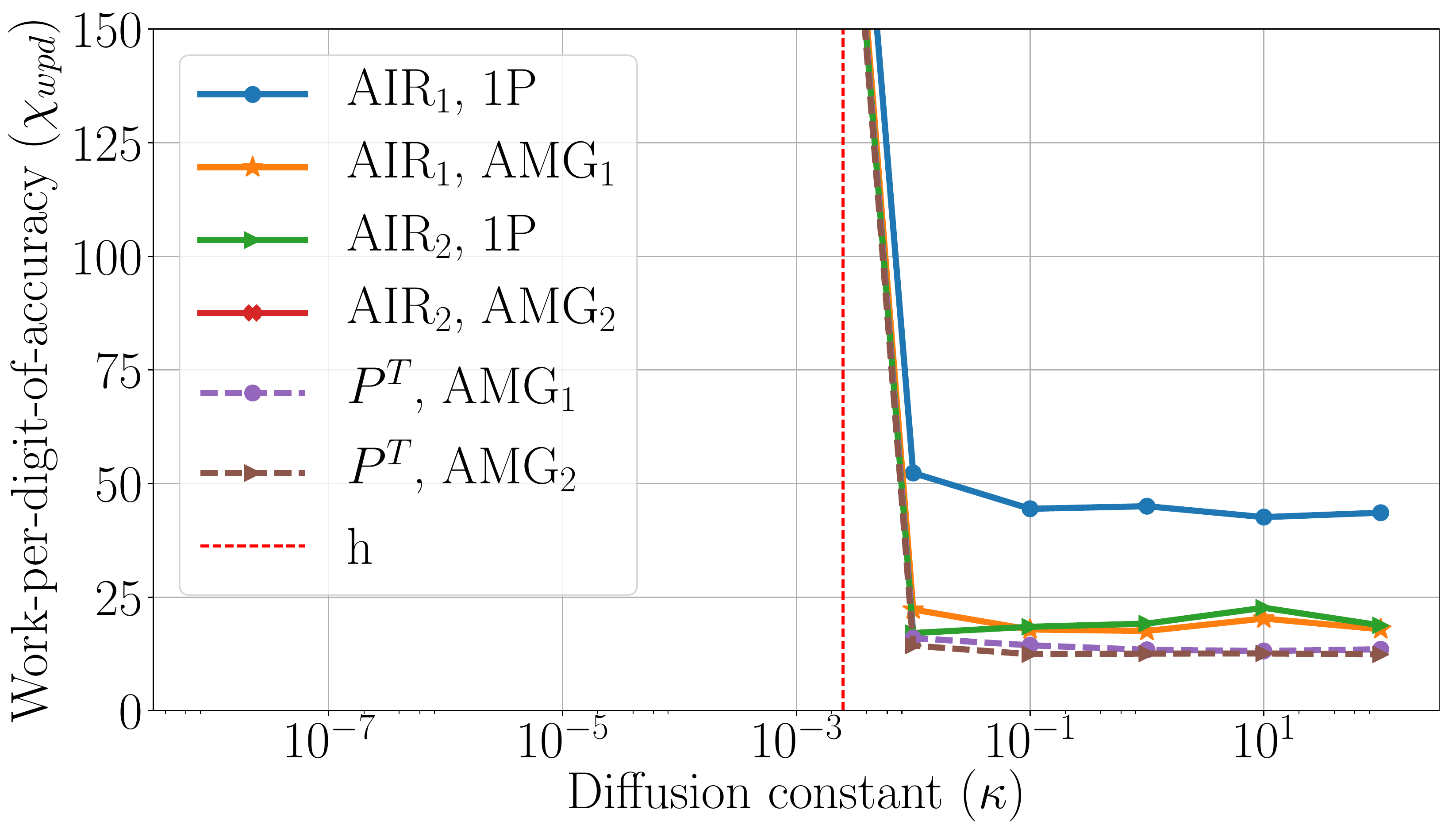}
    \caption{Degree-three finite elements}
  \end{subfigure} 
    \caption{WPD for $\ell$AIR and classical AMG applied to an SUPG discretization of a recirculating flow on
  	an unstructured mesh, using degree 1--3 finite elements. Spatial resolution is given by $h = \frac{1}{1250}, \frac{1}{625},\frac{1}{400}$,
	respectively, leading to $\approx 2\cdot 10^6$ DOFs for each problem. Classical AMG
	($R := P^T$) results are shown in a dotted line, and variations in $\ell$AIR in solid lines. As a reference, the typical
	convergence factor for $\chi_{wpd}$ of 20--25 WUs is $\rho\approx0.5$. }
\label{fig:supg_kappa}
\end{figure}

\subsubsection{Time-dependent recirculating flow}\label{sec:results:supg:time}

So far, the steady-state advection-diffusion equation has been considered which, in the case of a recirculating flow, is not well-posed
for the purely advective case. However, an alternative approach is to consider the time-dependent advection-diffusion equation:
\begin{align}
u_t - \nabla\cdot\kappa\nabla u + \beta\cdot\nabla u & = f\hspace{3ex} \text{ in }\Omega,\label{eq:supg_spacetime}
\end{align}
with spatial boundary conditions as before and some initial condition, $u = u_0$ at time $t = 0$. As an example, consider using a first-order
implicit backward Euler discretization in time and SUPG in space. Let $\mathcal{S}$ denote the discrete matrix associated with an SUPG
spatial discretization of \eqref{eq:supg_spacetime} and $u^{(i)}$ denote the solution at the $i$th time step with step size $\delta t$. Each time step
then consists of solving the linear system
\begin{align}
(I + \delta t \mathcal{S})u^{(i+1)} = u^{(i)} + \delta tf. \label{eq:supg_timestep}
\end{align}

Table \ref{tab:time_step} shows the average convergence factor of $\ell$AIR as applied to \eqref{eq:supg_timestep} for various parameter choices.
Here, $\ell$AIR proves to be an effective solver for implicit time-stepping of a recirculating flow, even with large time steps, pure advection, and
higher-order finite elements. In practice, considerations must be taken for an appropriate combination of spatial and temporal discretization
(e.g., CFL condition, $h$ vs. $\delta t$, etc.); nevertheless, results indicate that $\ell$AIR is a fast and robust solver for discretizations using
implicit time-stepping. Note that, for example, $\delta_t = dx^2$ in the advective case is a trivial system to solve because the linear system
is diagonally dominant like $1/dx$, but such results are included for completeness. Interestingly with $\ell$AIR, time-stepping in the
advection-dominated regime is actually faster in all cases than the diffusion-dominated regime, likely due to the conditioning of an
advective matrix scaling like $1/dx$ and diffusive like $1/dx^2$.

{
\begin{table}[!h]
\centering
\begin{tabular}{| c || c c c | c c c | c c c| c c c |}\Xhline{1.25pt}
& \multicolumn{6}{c|}{Degree-one elements} & \multicolumn{6}{c|}{Degree-two elements} \\\hline
$\kappa$ & \multicolumn{3}{c|}{0} & \multicolumn{3}{c|}{1} & \multicolumn{3}{c|}{0} & \multicolumn{3}{c|}{1} \\\hline
$\delta t$ & $dx^2$ & $dx$ & $\sqrt{dx}$ & $dx^2$ & $dx$ & $\sqrt{dx}$ & $dx^2$ & $dx$ & $\sqrt{dx}$ & $dx^2$ & $dx$ & $\sqrt{dx}$ \\ \hline
CF & $10^{-8}$ & 0.01 & 0.29 & 0.1 & 0.38 & 0.49 & $10^{-8}$ & 0.01 & 0.27 & 0.14 & 0.36 & 0.45 \\
WPD & 0.75 & 3.0 & 12.5 & 5.8 & 13.1 & 17.4 & 0.74 & 3.3 & 11.0 & 6.2 & 11.6 & 14.8 \\\Xhline{1.25pt}
\end{tabular} 
\caption{Convergence factor (CF) and WPD of $\ell$AIR applied to \eqref{eq:supg_timestep} with various choices of time step, 
diffusion coefficient, and initial guess, for $dx = \frac{1}{1000}$ and about $1.25\cdot10^6$ DOFs.
}
\label{tab:time_step}
\end{table}
}

\subsection{Upwind-DG and advection-diffusion-reaction}\label{sec:results:dgu}

\begin{figure}[!h]
  \centering
  \begin{subfigure}[b]{0.3\textwidth}
\includegraphics[width=\textwidth]{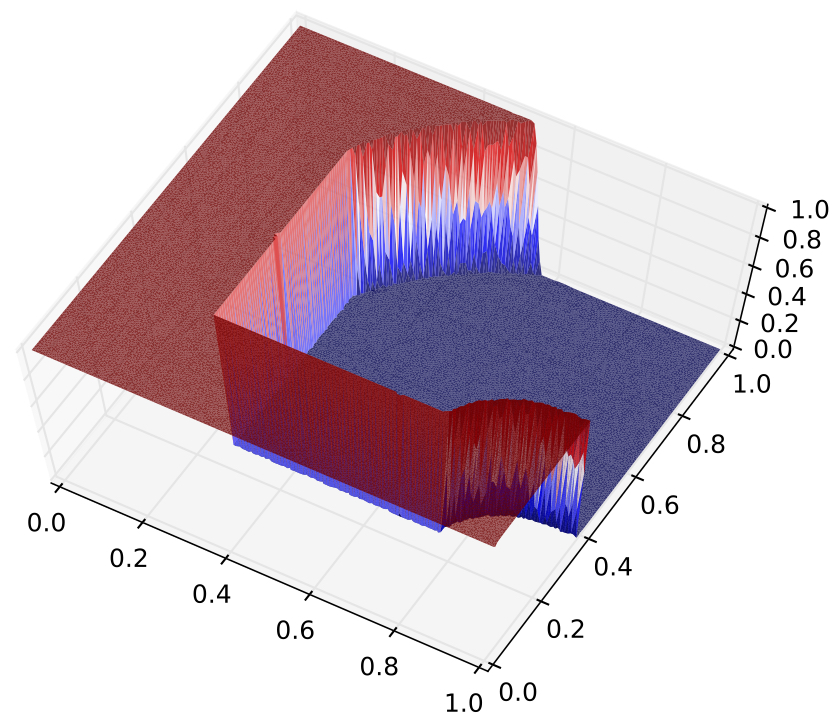}
    \caption{$\kappa = 0$}
  \end{subfigure}
  \begin{subfigure}[b]{0.3\textwidth}
\includegraphics[width=\textwidth]{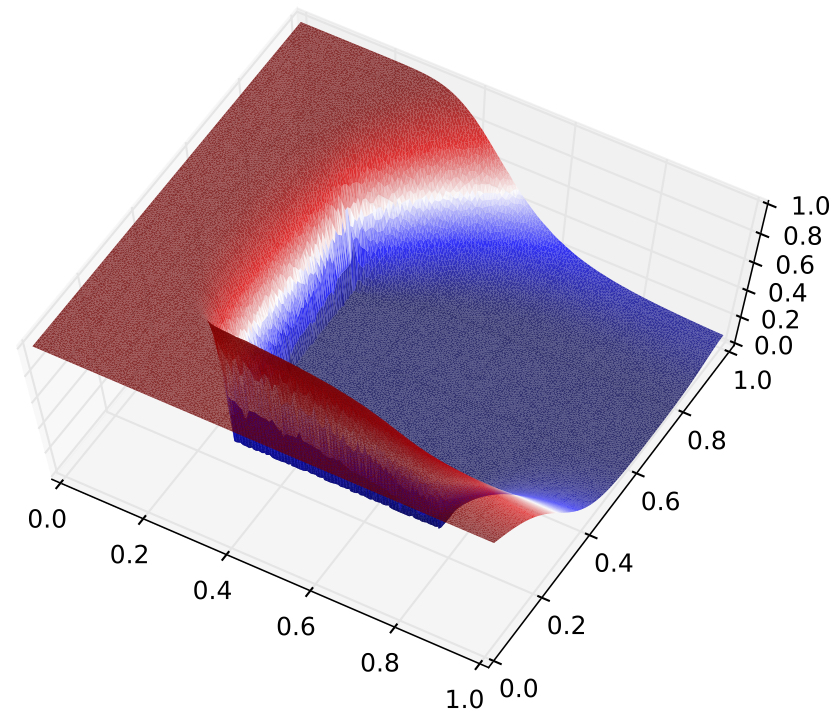}
    \caption{$\kappa = h = 0.005$}
  \end{subfigure}
  \begin{subfigure}[b]{0.3\textwidth}
\includegraphics[width=\textwidth]{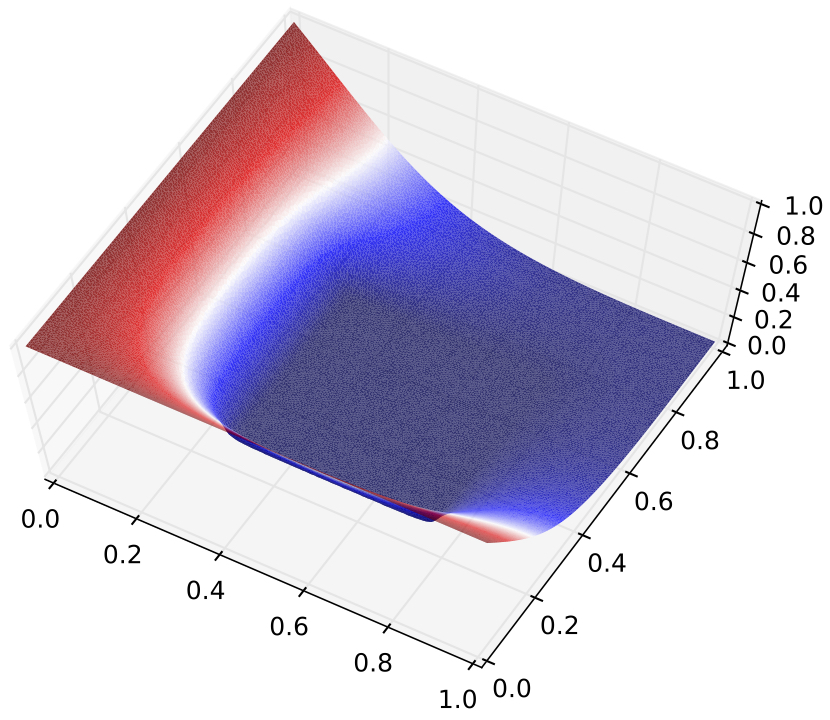}
    \caption{$\kappa = 1$}
  \end{subfigure}
  \caption{Solution of DGu discretization of the specified advection-diffusion-reaction equation for varying
  		diffusion coefficients, $\kappa \in\{0,h,1\}$, representing the purely advective, equal advection
		and diffusion, and diffusion-dominated cases, respectively.}
\label{fig:dg_adv_diff_re}
\end{figure}

The second case of advection-diffusion-reaction to be considered is a more general flow problem, written in conservative form
\cite{Ayuso:2009it,Sun:2005hm}. Let $\Gamma = \partial\Omega$ be the boundary of our domain, with inflow boundary
$\Gamma_{in} = \{x\in\Gamma \text{ $|$ } \beta(x)\cdot\mathbf{n}(x) < 0\}$ and outflow boundary 
$\Gamma_{out} = \{x\in\Gamma \text{ $|$ } \beta(x)\cdot\mathbf{n}(x) \geq 0\}$. Then, consider
\begin{align*}
\nabla\cdot \mathbf{\sigma}(u) + \gamma u & = f \hspace{4.5ex}\text{in }\Omega \\
u & = g_D \hspace{3ex}\text{on }\Gamma_{in} \\
- \kappa\nabla u\cdot\mathbf{n} & = g_N \hspace{3ex}\text{on }\Gamma_{out},
\end{align*}
where $\mathbf{\sigma} := -\kappa\nabla u+ \beta u$ is the physical flux. Additional assumptions that $\beta(x)$ has
no closed curves and that $|\beta(x)|\neq 0$ for all $x\in\Omega$ are made to ensure that the hyperbolic limit of steady-state transport for $\kappa = 0$
is well-posed. In the spirit of steady-state transport, $\gamma$ represents the \textit{total cross section} and is taken to be
piecewise constant over the domain, varying by multiple orders of magnitude, and representing the thickness of the background material. 
Specifically, let the domain be $\Omega = (0,1)\times(0,1)$, with $\Gamma_D = \Gamma^{-}$ being the south and west boundaries
and $\Gamma_N = \Gamma^{+}$ the north and east boundaries. Then, let $g_N(x,y) = 0$, $g_D(x,y) = 1$, and
\begin{align*}
\gamma(x,y) = \begin{cases} 10^{4} & x,y\in(0.25,0.75) \\ 10^{-4} & \text{otherwise}\end{cases}, \hspace{4ex}
\beta(x,y) =  \Big( y^2, \cos(\sfrac{\pi x}{2})^2 \Big).
\end{align*}
Such choices correspond to a curved velocity field, facing straight north at $y=0$ and straight east at $x = 1$, with
a total cross section that is thick ($\gamma \gg 1$) in a block in the center of the domain and thin ($\gamma \ll 1$)
outside of this block. An upwind discontinuous Galerkin (DGu) formulation is used to discretize (Eq. 3.4 in \cite{Ayuso:2009it})
on an unstructured mesh, and the solution for varying levels of diffusion is shown in Figure \ref{fig:dg_adv_diff_re}.

\tcr{Here, the highly nonsymmetric (block triangular) hyperbolic limit, in which classical AMG proves unable to converge,
is considered. In order to compare with other solvers specifically developed for such problems, $\ell$AIR is compared with
a nonsymmetric smoothed aggregation (NSA) algorithm \cite{Sala:2008cv}. Although NSA is able to solve most problems
that are addressed here, in terms of convergence factor and WPD, $\ell$AIR proves to be superior.} Note that geometric
multigrid (GMG), when applicable, is often faster than all algebraic approaches. However, here we are working on
unstructured meshes and without line relaxation, which makes comparisons with GMG difficult.

\subsubsection{DG block structure}\label{sec:results:dgu:block}

One of the unique features of a DG discretization is the inherent block structure associated with it, where each element
comprises a set or block of DOFs in the matrix. In \cite{amgir}, the resulting linear system was scaled by the block-diagonal
inverse, $A\mathbf{x} = \mathbf{b} \mapsto D_B^{-1}A\mathbf{x}=D_B^{-1}\mathbf{b}$, where $D_B$ is the block-diagonal
of $A$, in order to maintain a lower-triangular structure for the Neumann series approximation to $A_{ff}^{-1}$. In adding
diffusion (that is, symmetric components to the matrix) and using generalized $\ell$AIR that does not depend on a triangular
structure to the matrix, it is not obvious that scaling by the block-diagonal inverse, or generally utilizing the block structure,
is important. However, using the block structure significantly improves convergence of $\ell$AIR for all advection-dominated
problems. One possible explanation is the effect of block-diagonal scaling on the condition number of the matrix:\footnote{Due
to the significant difference in conditioning of $A$ and $D_B^{-1}A$, a detailed analysis connecting the finite element theory and
linear algebra in this regard is likely in order, but outside the scope of this paper.} 

For advection-dominated problems with large discontinuities in $\gamma$, scaling by the block-diagonal inverse maps a near-singular
matrix to one conditioned
about as we would expect for a pure advection problem ($\frac{1}{h}$). Such improvement of conditioning of the linear system is good
for iterative solvers in general. However, the importance to $\ell$AIR is specifically that if $D_B^{-1}A$ is well-conditioned, then we should
be able to pick F-points such that the resulting submatrix of F-F connections is also well-conditioned. Then (i) we should be
able to form a good approximation to $-A_{cf}A_{ff}^{-1}$ in ideal restriction, and (ii) F-relaxation should converge well. The
effect of scaling by $D_B^{-1}$ on relaxation can be seen in Table \ref{tab:scale_relaxation}, which shows the average 
convergence factor of Jacobi relaxation on the larger matrices, $A$ and $D_B^{-1}A$, and the F-F submatrices, $A_{ff}$
and $(D_B^{-1}A)_{ff}$. In the advection-dominated case, Jacobi relaxation does not converge on the submatrix $A_{ff}$,
and diverges as applied to $A$. After scaling by $D_B^{-1}$, the pure advection case achieves convergence factors of
$\approx 0.1$ on the F-F block. 

\begin{table}[!h]
\centering
\begin{tabular}{| c | c c c c c c c c c |}\Xhline{1.25pt}
$\kappa$ & 0 & $10^{-6}$ & $10^{-4}$ & 0.001 &  0.01 &  0.1 & 1 & 10 & 100 \\\hline
cond$(A)$ & 1164172 &1074536 & 187700 & 92619 & 37836 & 7422 & 1132 & 413 & 684 \\ 
cond$(D_B^{-1}A)$ & 86.7 & 86.7 & 85.2 & 77.9 & 125 & 228 & 296 & 359 & 619 \\\Xhline{1.25pt}
\end{tabular} 
\caption{Condition number of $A$ and $D_B^{-1}A$ as a function of diffusion coefficient $\kappa$ for degree-one finite elements
	on an unstructured mesh, with $h\approx \frac{1}{25}$ and 3030 DOFs (computed using NumPy \cite{numpy}).}
\label{tab:condition}
\end{table}
\begin{table}[!th]
\centering
\begin{tabular}{| c | c c c c c c c c c c c c |}\Xhline{1.25pt}
$\kappa$ & 0 & $10^{-8}$ & $10^{-7}$ & $10^{-6}$ & $10^{-5}$ & $10^{-4}$ & 0.001 &  0.01 &  0.1 & 1 & 10 & 100 \\\hline
CF$(A)$ & 1.58 & 1.58 & 1.58 & 1.57 & 1.54 & 1.28 & 0.99 & 0.97 & 0.94 & 0.94 & 0.94 & 0.93 \\
CF$(D_B^{-1}A)$ & 0.95 & 0.95 & 0.95 & 0.95 & 0.94 & 0.91 & 0.89 & 0.88 & 0.89 & 0.89 & 0.88 & 0.88 \\
CF$(A_{ff})$ & 1.00 & 1.00 & 1.00 & 1.00 & 0.98 & 0.94 & 0.83 & 0.65 & 0.67 & 0.70 & 0.71 & 0.73 \\
CF$((D_B^{-1}A)_{ff})$ & 0.12 & 0.33 & 0.38 & 0.47 & 0.47 & 0.51 & 0.52 & 0.49 & 0.54 & 0.50 & 0.50 & 0.50 \\\Xhline{1.25pt}
\end{tabular} 
\caption{Average convergence factor of 50 iterations of Jacobi relaxation on $A$, $A_{ff}$, $D_B^{-1}A$, and $(D_B^{-1}A)_{ff}$.
	Degree-one finite elements are used on an unstructured mesh, with $h= \frac{1}{500}$, approximately $2\cdot 10^6$ DOFs,
	and $\kappa \in[0,100]$.}
\label{tab:scale_relaxation}
\end{table}

Tables \ref{tab:condition} and \ref{tab:scale_relaxation} indicate that the DG block-structure can be an important feature for a solver to
consider, particularly for advection-dominated problems. However, block-matrix structure can be handled in a number of ways. Three
natural approaches are (i) treating the matrix as is, without considering block structure, (ii) scaling the system by the block-diagonal
inverse, and (iii) treating the entire AMG hierarchy nodally, that is, computing the SOC, CF-splitting, and transfer operators by block. 
One benefit of $\ell$AIR is, similar to smoothed aggregation (SA) multigrid methods, there is a natural approach to handling block structure
by coarsening by block and building $R$ by block. This is in contrast to classical AMG methods, which typically struggle with matrices
of block structure. If the block structure is not accounted for, no combinations of $\ell$AIR and AMG converge in the advection-dominated
case, which is the focus of this paper. For advection-dominated problems, treating the system nodally leads to similar convergence
factors to those obtained on the block-diagonally scaled matrix, with slightly higher setup wall-clock times because a block neighborhood
in $\ell$AIR is larger than a scalar neighborhood and, thus, the dense linear systems larger. However, scaling by the block inverse
is detrimental to convergence in the diffusion-dominated case, particularly for high-order finite elements, likely because it is mapping
a near-symmetric matrix to be highly nonsymmetric. For that reason, the remainder of this section uses block $\ell$AIR as the most effective
general solver. Because block classical AMG interpolation methods are not well developed, one-point interpolation is used, which
naturally extends to the block setting. 

\begin{remark}
With careful tuning of parameters, scalar, classical AMG methods do converge on most problems that are strongly diffusion dominated,
$\kappa \gg 1/h$. However, such results were found to be very sensitive to parameter tuning, particularly for increased finite-element
order and problem dimension, and typically did not outperform NSA or $\ell$AIR. Classical AMG was unable to converge for all
advection-dominated problems tested, except for linear finite elements in two dimensions, with $\kappa \approx 1/h$, where
convergence factors were still worse than $\ell$AIR or NSA. 
\end{remark}

\subsubsection{$\ell$AIR and convergence as a function of $\kappa$}\label{sec:results:dgu:kappa}

Figure \ref{fig:dg_kappa} shows WPD of block-$\ell$AIR and NSA-with-block-diagonal-scaling applied to the linear systems
corresponding to $\kappa \in[10^{-10},100]$. $\ell$AIR performs well in all advection-dominated cases, including high-order
elements and two and three dimensions, and is $3-8\times$ faster than NSA in terms of WPD for parameters tested
here.\footnote{Moderate optimization of parameters has been done for both solvers; it is possible results could be
improved in either case, in particular if focused on a specific problem} In all cases, performance of block NSA proved to
be worse than NSA, and results are not included. It is also interesting to note that applying scalar $\ell$AIR to the block-diagonally
scaled system performs as well or better in highly advection-dominated cases. However, as diffusion is introduced, even for 
$\kappa > 10^{-5}$, convergence of scalar $\ell$AIR degrades substantially with high-order elements and three dimensions. 

Note that, using a block solver or not, diffusion-dominated problems in 3d appear to be difficult for $\ell$AIR and SA. That
high-order and high-dimensionality DG discretizations of elliptic problems can be difficult for solvers and AMG is well known. 
This has prompted research into specific modifications for an effective DG solver; in particular, the SA approach developed
for high-order DG discretizations of elliptic problems in \cite{Olson:2011ju} may outperform $\ell$AIR and SA as studied here 
for the 3d diffusion-dominated case. However, in \cite{Olson:2011ju}, W-cycles were necessary for scalable convergence,
again indicating the difficulty of such discretizations. Interestingly, it seems that DG discretizations of advection-dominated
problems can now be solved faster than diffusion-dominated problems, contrary to traditional thought on fast solvers and
classes of PDEs. 

\begin{figure}[!b]
  \centering
  \captionsetup[subfigure]{justification=centering}
  \begin{subfigure}[b]{0.45\textwidth}
\includegraphics[width=\textwidth]{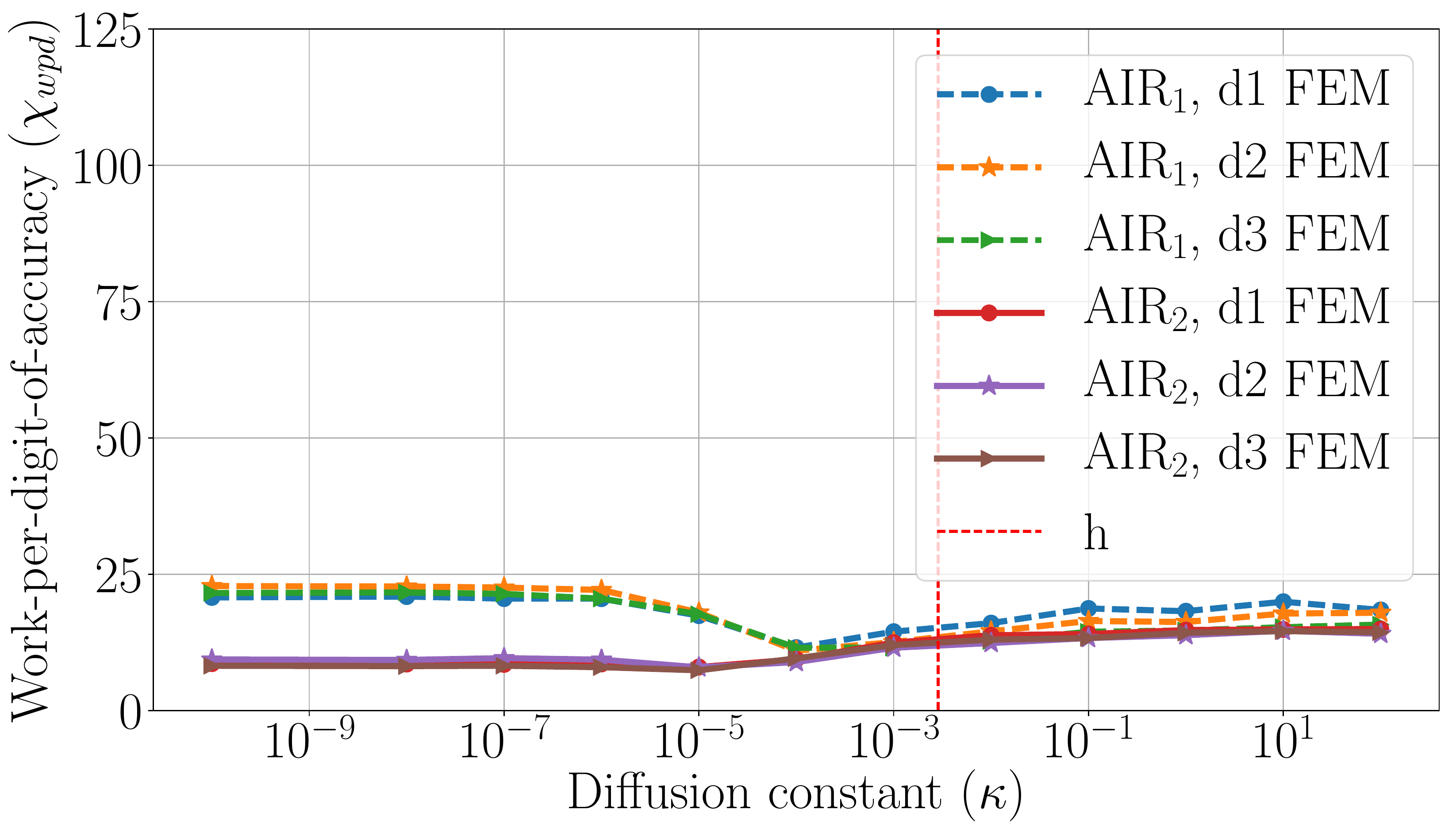}
    \caption{Two dimensions, $\ell$AIR}
   \label{subfig:dg_d1d1}
  \end{subfigure}
  \begin{subfigure}[b]{0.45\textwidth}
\includegraphics[width=\textwidth]{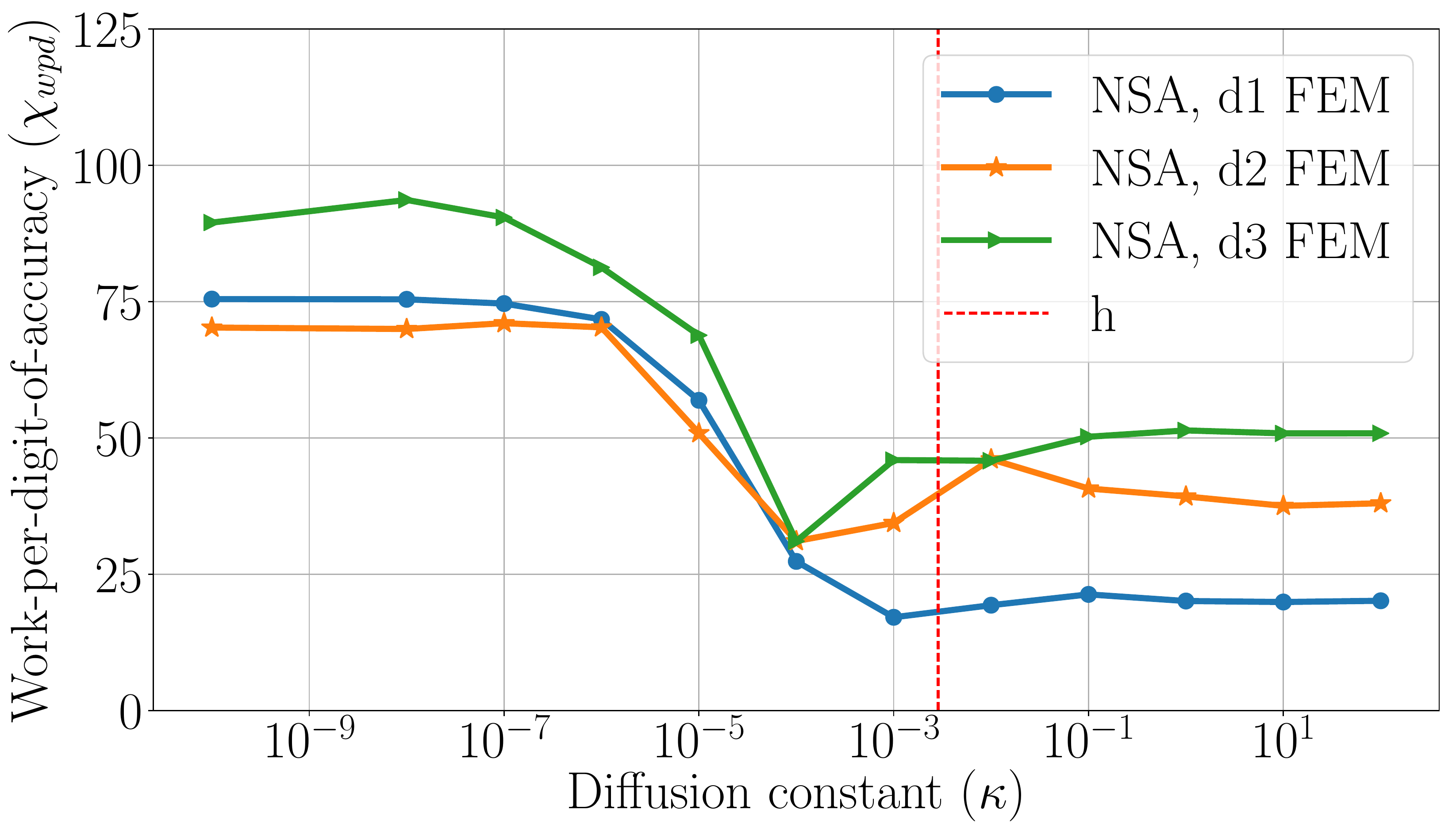}
    \caption{Two dimensions, NSA}
   \label{subfig:dg_d1d2}
  \end{subfigure}
      \begin{subfigure}[b]{0.45\textwidth}
\includegraphics[width=\textwidth]{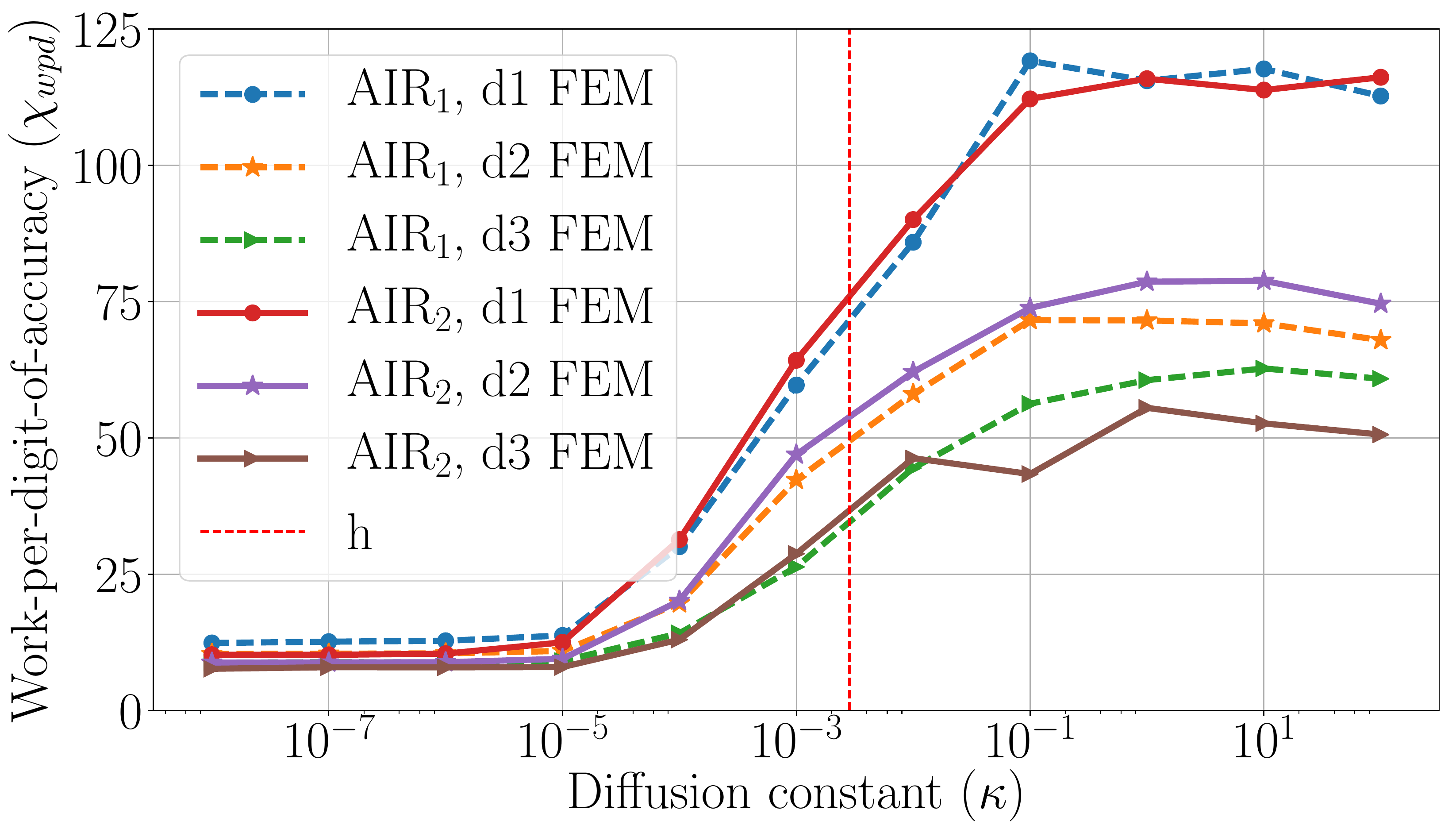}
    \caption{Three dimensions, $\ell$AIR}
     \label{subfig:dg_d2}
\end{subfigure}
  \begin{subfigure}[b]{0.45\textwidth}
\includegraphics[width=\textwidth]{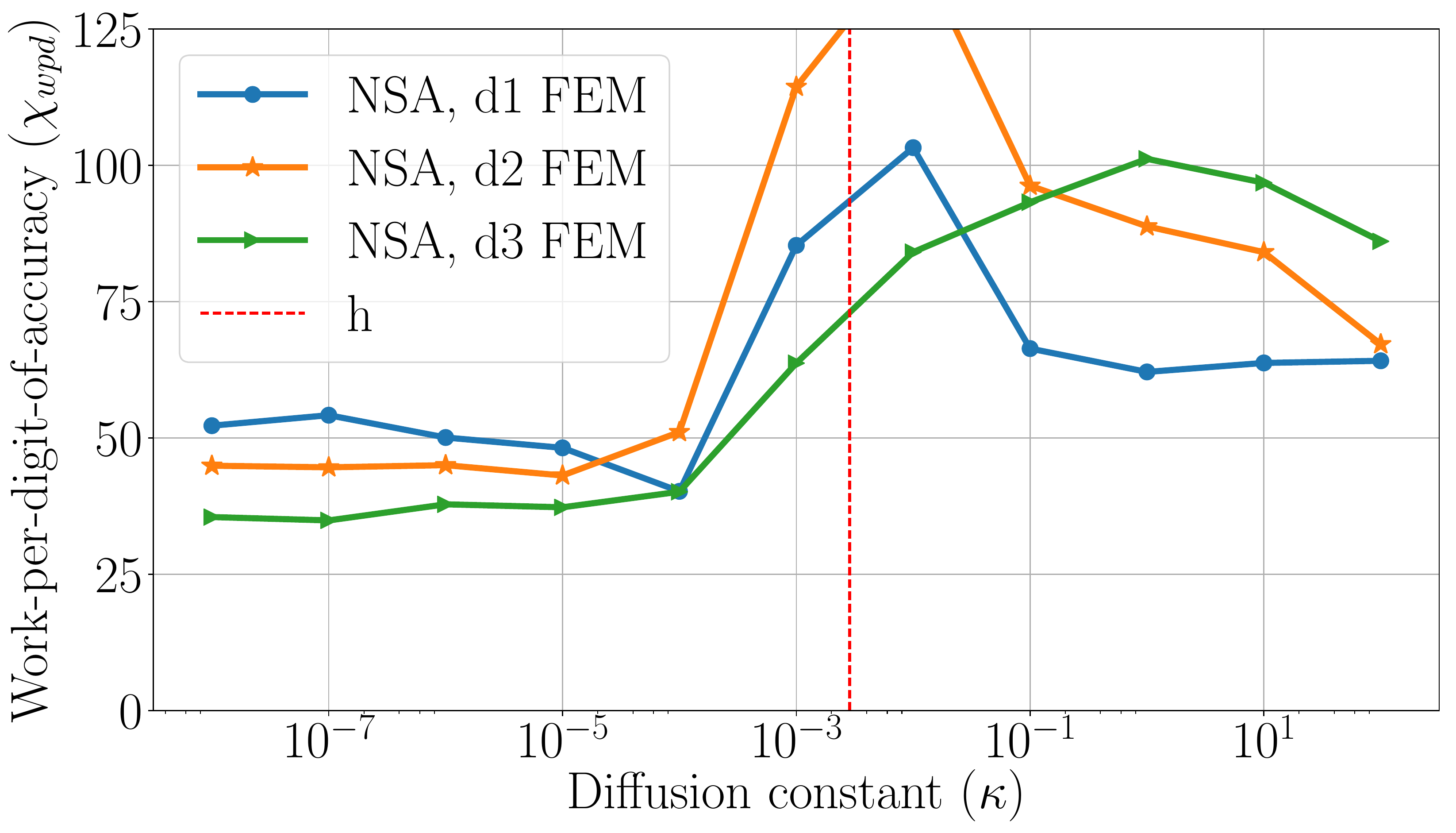}
    \caption{Three dimensions, NSA}
   \label{subfig:dg_d3}
  \end{subfigure}
  \caption{WPD for $\ell$AIR and classical AMG applied to a DGu discretization of advection-diffusion-reaction on
  	an unstructured mesh and degree 1--3 finite elements. The spatial resolution for high-order finite element matrices was
	chosen so that all matrices have approximately the same number of degrees of freedom as the linear element matrix, that
	is, $h = \frac{1}{500}, \frac{1}{350},\frac{1}{275}$, respectively, leading to $\approx 2\cdot 10^6$ DOFs for each problem.}
\label{fig:dg_kappa}
\end{figure}

\section{Overview and conclusions}\label{sec:conc}

Although nonsymmetric linear systems arise often in the study of numerical PDEs, they tend to lack fast and robust solvers.
AMG is often the solver of choice for SPD matrices in high-performance codes and,
here, we present a new variation of AMG based on a local approximation to the ideal restriction operator. The resulting method, $\ell$AIR,
proves to be a fast and robust solver for scalar advection-diffusion-reaction equations. For diffusion-dominated problems, $\ell$AIR is
competitive with classical AMG techniques and, using an SUPG discretization, is able to solve a steady-state recirculating flow problem
with high-order finite elements and an unstructured mesh. In the advection-dominated case, the steady-state recirculating flow is not
well-posed; however, $\ell$AIR proves to be a robust solver for implicit time-stepping applied to a recirculating flow as well as an upwind
DG discretization of advection-diffusion-reaction. For the advection-diffusion-reaction problem discretized by DG,
$\ell$AIR is able to solve high-order discretizations, from strictly advective to diffusion dominated, on unstructured grids, and with
moderate complexity (see Figure \ref{subfig:dg_d1d1}). $\ell$AIR also
consistently outperforms a nonsymmetric smoothed aggregation algorithm, the current state-of-the-art for highly nonsymmetric
problems, in the advection-dominated regime.

A block analysis of AMG in Section \ref{sec:background} provided theoretical motivation for $\ell$AIR. A more recent, companion
paper \cite{amgir} extends this analysis, developing a complete convergence framework for $\ell$AIR, with sufficient conditions for
the $\ell^2$-convergence of error and residual. Building on the convergence framework in \cite{amgir} and encouraging results
shown here, a future research direction of interest is extending $\ell$AIR to systems of PDEs. Systems of PDEs remain something
of an open question for AMG in general; $\ell$AIR provides a robust method for solving scalar advection-dominated problems that
classical AMG struggles with, and this could be of great use for solving systems with strong advection. A few notable systems of
interest include the Euler equations, shallow-water equations, and Navier Stokes.

\bibliographystyle{siamplain}
\bibliography{air.bib}

\end{document}